\documentclass[11pt]{amsart} 
\usepackage{amssymb,amsmath,amsthm,amsfonts,mathrsfs}
\usepackage{graphicx,graphics,xcolor,latexsym,comment,esint,tikz,tikz-cd,times}
\usepackage{pgfplots}
\usepackage{stmaryrd}
\usetikzlibrary{shapes.geometric,positioning}
\usepackage{comment} 
\excludecomment{codes}
\usepackage[OT1]{fontenc}
\usepackage{mathtools}
\excludecomment{codes}
\usetikzlibrary{arrows,intersections}

\usepackage{filecontents}
\usepackage{url}
\usepackage[colorlinks = true,
            linkcolor = blue,
            filecolor = magenta,
            urlcolor = blue,
            citecolor = magenta,
            pdftitle = {Comments},
           ]{hyperref}

\usepackage[shortlabels]{enumitem}
\setlist[itemize]{align = parleft,left = 4pt..1.5em}
\setlist[enumerate]{align = parleft,left = 4pt..2.0em}
\usepackage{float}
\usepackage[english]{babel}
\usepackage[strict = true,style = english]{csquotes}
\usepackage[backend = biber,
						style = alphabetic,
						natbib = true,
            sorting = nyt,
            maxbibnames = 99,
            abbreviate = true]{biblatex}

\AtEveryBibitem{
 \clearfield{url}
 \clearfield{issn}
 \clearlist{location}
 \clearfield{month}
 \clearfield{series}

 \ifentrytype{book}{}{
  \clearlist{publisher}
  \clearname{editor}
 }
}

\newcommand{\beq}{\begin{equation}}
\newcommand{\eeq}{\end{equation}}
\newcommand{\bea}{\begin{eqnarray}}
\newcommand{\eea}{\end{eqnarray}}
\newcommand{\beas}{\begin{eqnarray*}}
\newcommand{\eeas}{\end{eqnarray*}}

\newcommand{\R}{\mathbb{R}}
\renewcommand{\P}{\mathbb{P}}
\newcommand{\E}{\mathbb{E}}
\newcommand{\myMod}[1]{\left\llbracket #1\right\rrbracket}
\newcommand{\myDist}[1]{\mathop{\mathrm{dist}}\!\left(#1\right)}

\newcommand{\Norm}[1]{\left|\left|  #1   \right|\right|}
\newcommand{\InPrd}[1]{\left\langle #1 \right\rangle}

\newcommand{\dom}{\mathbb{T}}
\newcommand{\ud}{\ensuremath{\mathrm{d}}}

\newtheorem{innercustomgeneric}{\customgenericname}
\providecommand{\customgenericname}{}
\newcommand{\newcustomtheorem}[2]{%
  \newenvironment{#1}[1]
  {%
   \renewcommand\customgenericname{#2}%
   \renewcommand\theinnercustomgeneric{##1}%
   \innercustomgeneric
  }
  {\endinnercustomgeneric}
}

\newcustomtheorem{customthm}{Theorem}
\newcustomtheorem{customlemma}{Lemma}
\theoremstyle{plain}
\newtheorem{theorem}{Theorem}[section]

\newtheorem{lemma}[theorem]{Lemma}
\newtheorem{proposition}[theorem]{Proposition}

\theoremstyle{definition}
\newtheorem{definition}[theorem]{Definition}

\newtheorem{assumption}{Assumption}

\newtheorem{remark}[theorem]{Remark}

\newtheorem{example}[theorem]{Example}

\numberwithin{hypothesis}{section}
\numberwithin{assumption}{section}

\def\me{{\mathbb  E}}

\def\mr{{\mathbb  R}}

\def\mz{{\mathbb  Z}}

\def\mt{{\mathbb T}}

\title[]{Parabolic Anderson model with colored noise on torus}

\author[L.~Chen]{Le Chen}
\address{L.~Chen: Department of Mathematics and Statistics, Auburn, AL, 36849, USA.}
\email{\url{le.chen@auburn.edu}}

\author{Cheng Ouyang}
\address{C.~Ouyang: Department of Mathematics, Statistics, and Computer Science, University of Illinois at Chicago, Chicago, IL, 60607, USA.}
\email{\url{couyang@uic.edu}}

\author{William Vickery}
\address{W.~Vickery: Department of Mathematics, Statistics, and Computer Science, University of Illinois at Chicago, Chicago, IL, 60607, USA.}
\email{\url{wvicke2@uic.edu}}

\subjclass[2010]{Primary 60H15, Secondary: 60G60,~37H15}

\keywords{
  Stochastic heat equation on torus,
  Dalang's condition,
  measure-valued initial condition,
  Brownian bridge,
  moment asymptotics,
  intermittency,
  moment Lyapunov exponent,
  theta function.
}

\begin{document}
\maketitle

\begin{abstract}
  We construct an intrinsic family of Gaussian noises on $d$-dimensional flat
  torus $\mathbb{T}^d$. It is the analogue of the colored noise on
  $\mathbb{R}^d$, and allows us to study stochastic PDEs on torus in the It\^{o}
  sense in high dimensions. With this noise, we consider the \textit{parabolic
  Anderson model} (PAM) with measure-valued initial conditions and establish
  some basic properties of the solution, including a sharp upper and lower bound
  for the moments and H\"{o}lder continuity in space and time. The study of the
  toy model of $\mathbb{T}^d$ in the present paper is a first step towards our
  effort in understanding how geometry and topology play an role in the behavior
  of stochastic PDEs on general (compact) manifolds.
\end{abstract}

\tableofcontents

\section{Introduction}

In this paper, we construct an intrinsic family of Gaussian noises on
$d$-dimensional torus $\mathbb{T}^d \coloneqq [-\pi,\pi]^d$ that is
\textit{colored} in space and white in time. It is the analogue of the colored
noise on $\mathbb{R}^d$ and it enables one to study, in It\^{o}'s sense,  the
\textit{parabolic Anderson model} (PAM, see~\eqref{E:PAM} below) and other
\textit{stochastic partial differential equations} (SPDEs) on $\mathbb{T}^d$ in
higher dimensions. In this setting, we aim at understanding how topology and
geometry of non-Euclidean spaces influence the behavior of the solution.

More specifically, let $G(t,x)$ be the heat kernel on $\mathbb{T}^d$, i.e.,
\begin{align}\label{E:HeatKernel}
  G(t,x) \coloneqq \sum_{k\in \mathbb{Z}^d} \prod_{i=1}^d p\left(t,x_i + 2\pi k_i\right)
  \quad \text{for all $t>0$ and $x = (x_1,\cdots, x_d) \in \mathbb{T}^d$,}
\end{align}
where we use $p(t,x) \coloneqq (2\pi t)^{-1/2} e^{-x^2/(2t)}$ for $x\in\R$ to
denote the heat kernel on $\R$. Sometimes, we use a subscript $d$ to denote the
dimension of the space, e.g., $p_{d}(t,x) \coloneqq (2\pi t)^{-d/2}
e^{-\frac{|x|^2}{2t}}$ for $x\in\R^d$ where $|x|=\sqrt{x_1^2+\cdots+x_d^2}$.
Alternatively, $G_d(t,x)$ in~\eqref{E:HeatKernel} can be expressed as
\begin{align*}
  G_d(t,x)= \prod_{i=1}^{d} G_1(t,x_i).
\end{align*}
When there is no confusion from context, this subscript $d$ will be omitted.
\bigskip

Fix $\rho\geq0$ and $\alpha>0$. The spatial covariance function of our colored
noise will be given by
\begin{gather}\label{E:f}
  f_{\alpha,\rho}(x) \coloneqq \frac{\rho}{(2\pi)^d} + \frac{1}{\Gamma(\alpha)}\int_0^\infty t^{\alpha-1}\left(G(t,x)-\frac{1}{(2\pi)^d}\right) \ud t, \quad \text{for $x\in \mathbb{T}^d$.}
\end{gather}
The function $f_{\alpha,\rho}$ is an analogue of the Riesz kernel on
$\mathbb{R}^d$, and is related to the spectral decomposition of the Laplacian on
$\mathbb{T}^d$ (see Section~\ref{S:Noise} for more details). Moreover, it is
well understood that for all $\alpha>0$ and $\rho\ge 0$, there exists some
constant $C>0$ such that $f_{\alpha,\rho}$ admits the following estimate (see,
e.g., Lemma~2.9 of~\cite{brosamler:83:laws})
\begin{align} \label{E:fEst}
  \left|f_{\alpha,\rho}(x)\right| \leq
  \begin{dcases}
    C                          & \text{if $\alpha>d/2$} \\
    C\left(1+\log^- |x|\right) & \text{if $\alpha=d/2$} \\
    C|x|^{-d+2\alpha}          & \text{if $\alpha<d/2$}
  \end{dcases},
  \quad \text{for all $x\in \mathbb{T}^d$},
\end{align}
where $\log^- t \coloneqq \max(0,-\log t)$. The above estimate implies that the
colored noise is smoother than the white noise. In addition, the parameter
$\rho$ controls the level of $f_{\alpha,\rho}$ while $\alpha$ controls its
regularity. In what follows, we adopt the following convention:
\begin{align}\label{E:f_alpha}
  f_\alpha(x) \coloneqq f_{\alpha,0}(x) \quad \text{and} \quad
  f(x) \coloneqq f_{1}(x).
\end{align}
It is known that $f(x)$, the Green's function of the Laplace on $\mathbb{T}^d$,
is not positive. The same can be said to our covariance function
$f_{\alpha,\rho}(x)$ (see Lemma~\ref{L:Cov} below). Throughout the paper, we use
the convention that
\begin{align}\label{E:Convention}
  G(t,x,y) = G\left(t,\myMod{x-y}\right) \quad \text{and} \quad
  f_{\alpha,\rho}(x,y) = f_{\alpha,\rho}(\myMod{x-y}),
\end{align}
where for $x = (x_1,\cdots,x_d) \in\R^d$, $\myMod{x} \coloneqq
\left(\myMod{x_1},\cdots, \myMod{x_n}\right)$ with $\myMod{x_i} = \text{mod}(x_i
+ \pi, 2\pi) - \pi$, i.e., $\myMod{x_i}$ is the \textit{signed
remainder}\footnote{Here we use the same convention as \textit{Wolfram
  Mathematica} that the $\text{mod}$ function always returns the positive
  reminder, i.e., $\text{mod}(x,2\pi) = c$ iff $c\in [0,2\pi)$ and $x = 2\pi d +
  c$ for some $d\in \mathbb{Z}$. Alternatively, $\myMod{x}$ can be equivalently
expressed as $\myMod{x} \coloneqq \text{mod}(x,2\pi,-\pi)$, where
$\text{mod}(m,n,d)$ is the same modulo function with an offset $d$ as that in
\textit{Wolfram Mathematica}.} of $x_i$ when divided by $\pi$. Note that
$\left|\myMod{x-y}\right|$ is the distance between $x$ and $y$ on torus, namely,
$\myDist{x,y} \coloneqq \left|\myMod{x-y}\right|$.

With the colored noise in hand described as above, we consider the following
SPDE, or PAM, on $\mathbb{T}^d$
\begin{align}\label{E:PAM}
  \begin{dcases}
    \partial_t u(t,x) = \frac{1}{2}\triangle u(t,x)+ \lambda u(t,x) \dot{W}(t,x), & (t,x)\in (0,\infty)\times\mt^d,\\
    u(0,\cdot) = \mu,
  \end{dcases}
\end{align}
where $\dot{W}(t,x)$ is a centered Gaussian noise white in time and colored in
space, and $\lambda \ne 0$ is a constant that controls the level of the noise.
We assume that the initial condition $\mu$ is a finite (nonnegative) measure on
$\mathbb{T}^d$, namely,
\begin{align}\label{E:Cmu}
  C_\mu \coloneqq \int_{\mathbb{T}^d} \mu(\ud x) <\infty.
\end{align}

The solution to~\eqref{E:PAM} is interpreted as the integral equation or the
\textit{mild solution} (see Definition~\ref{D:Sol} below for more details)
\begin{equation}\label{E:mild-sol}
  u(t, x) = J_0(t,x) + \lambda \int_0^t \int_{\dom^d} G(t-s, x, y) u(s, y) W(\ud s, \ud y) \quad \text{a.s.},
\end{equation}
where the stochastic integral is in the It\^o/Walsh sense and $J_0(t,x)$ refers
to the solution to the homogeneous equation, namely,
\begin{align}\label{E:J0}
  J_0(t,x) \coloneqq \int_{\dom^d} G(t, x, y) \mu(\ud y).
\end{align}

\begin{assumption}[Dalang's condition]\label{A:Dalang}
  We assume that the correlation function $f_{\alpha,\rho}(\cdot)$ satisfies the
  \begin{align}\label{E:Dalang}
    \sum_{k\in \mz^d_*}\frac{\mathcal{F}(f_{\alpha,\rho})(k)}{|k|^{2}}<\infty
    \quad \Longleftrightarrow \quad
    2(\alpha + 1) > d \quad  \textnormal{(thanks to~\eqref{E:theta})},
  \end{align}
  where $\mz^d_*\coloneqq \mathbb{Z}^d \setminus \{0\}$.
\end{assumption}
\begin{remark}
  \textit{Dalang's condition} usually refers to the condition on the correlation
  function $f$ such that the corresponding stochastic partial differential
  equation (SPDE) with an additive noise has a unique solution;
  see~\cite{dalang:99:extending}. For most parabolic SPDEs, Dalang's condition
  is usually the necessary and sufficient condition for the existence of a
  unique solution when the noise is of multiplicative type. In the current
  setting, Dalang's condition~\eqref{E:Dalang} regulates the high frequencies of
  $f_{\alpha,\rho}$ in the Fourier mode, which in turn controls the singularity
  of $f_{\alpha,\rho}(x)$ at $x=0$ in the direct mode (see~\eqref{E:fEst}). We
  also would like to mention that Dalang's condition is only a condition on
  $\alpha$, but not $\rho$.
\end{remark}

\begin{remark}
  The solution theory of~\eqref{E:PAM} is rather straightforward when
  $\alpha>d/2$, since the covariance function $f_{\alpha,\rho}$ is bounded and
  continuous in this case (see~\eqref{E:fEst}). We hence assume $\alpha<d/2$ in
  the rest of our discussion.
\end{remark}

Our first main result is summarized in the following theorem:

\begin{theorem}\label{T:Main}
  Suppose that the correlation function $f_{\alpha,\rho}(\cdot)$ satisfies
  Dalang's condition~\eqref{E:Dalang} {for some $\alpha\in
  (0,d/2)$}. Then there exists a unique solution $u$ to~\eqref{E:PAM} starting
  from a finite measure $\mu$ on $\mathbb{T}^d$. Moreover, the solution
  satisfies the following properties
  \begin{enumerate}
    \item For all $t>0$ and $x, x'\in \dom^d$, it holds that
      \begin{align}\label{E:TwoPoint}
        \E\left(u(t,x)u(t,x')\right)
        = J_0(t,x)J_0(t,x') + \lambda^{-2} \iint_{\dom^{2d}}\mu(\ud z) \mu(\ud z') \mathcal{K}_\lambda(t-s, x, z, x', z')
      \end{align}
      where $\mathcal{K}$ is defined in Definition~\ref{D:Ln}.
    \item For all $t>0$, $x\in \dom^d$, and $p\ge 2$, it holds that
      \begin{align}\label{E:p-Moments}
        \Norm{u(t,x)}_p \le \sqrt{2} J_0(t,x) \left[H_{4 \lambda \sqrt{p}}(t)\right]^{1/2} ,
      \end{align}
      where the function $H_\lambda(t)$ is defined in~\eqref{E:Ht}. Moreover,
      when $\lambda^2 p$ is large enough, there exists some constant $C>0$ such
      that
      \begin{align}\label{E:p-Mom-Rate}
        \Norm{u(t,x)}_p \le C J_0(t,x) \exp\left(C \lambda^{\max\left(\frac{4}{2(1+\alpha)-d},\,2\right)} p^{\max\left(\frac{2}{2(1+\alpha)-d},\,1\right)}t\right),
      \end{align}
      for all $(t,x,p)\in (0,\infty)\times\dom^d\times [2,\infty)$.
  \end{enumerate}
\end{theorem}
\begin{remark}\label{R:novelty}
  It is worth mentioning that our result allows measure-valued initial
  conditions. In particular, it includes the commonly interested cases when
  $\mu=\delta_0$ and $\mu=1$. The proof of Theorem~\ref{T:Main} is inspired by
  the method developed in~\cite{chen.dalang:15:moments, chen.kim:19:nonlinear,
  chen.huang:19:comparison} for the \textit{stochastic heat equation} (SHE) on
  $\R^d$. The main novelty of our approach is to observe that the Brownian
  bridge (starting from $x$ and conditioned on arriving at $y$ at time $t$)
  plays an important role. The proof of the theorem is then based on the
  observation that a Brownian bridge on torus is comparable to that on a
  Euclidean space when $t$ is small, whereas it is comparable to a Brownian
  motion on torus when $t$ is large. This approach allows one to tackle the PAM
  in the direct mode (as opposed to the Fourier mode) and can be applied to
  study PAM on general manifolds. It circumvents the difficulty that on general
  manifolds, Fourier transform, or the spectral decomposition of the
  Laplace-Beltrami operator, is not particularly easy to work with.
\end{remark}
\smallskip

The study of PAM in current literature usually assumes that the covariance
function is positive. Unfortunately, it is not the case for $f_{\alpha,\rho}$ in
general.  However, one can show that it is indeed a positive function when
$\rho$ is large enough (see Lemma~\ref{L:Cov} below). In this case, we are able
to derive the following lower bound for the second moment of $u$.

\begin{theorem}\label{T:Lbd}
  Let the spatial covariance function $f: \mathbb{T}^{2d}\to \R_+$ be a generic
  nonnegative and nonnegative definite function. Assume that $f$ is uniformly
  bounded from below away from zero, i.e.,
  \begin{align}\label{E:Cf}
    C_f \coloneqq \inf_{x,x'\in \mathbb{T}^d} f(x,x') > 0,
  \end{align}
  then for all $\epsilon>0$, it holds that
  \begin{align}\label{E:Lbd}
    \E\left(u(t,x)^2\right) \ge J_0^2(t,x) + \frac{1}{2}\lambda^{-2} c_\epsilon^d \, C_\mu^2\, \exp\left(\frac{C_ft}{2}\right),
    \quad \forall (t,x)\in [\epsilon,\infty)\times \mathbb{T}^d,
  \end{align}
  where the constants $C_\mu$ and $c_\epsilon$ are given in~\eqref{E:Cmu}
  and~\eqref{E:C-Large}, respectively.
\end{theorem}

Comparing to Theorem~\ref{T:Main}, Theorem~\ref{T:Lbd} provides a matching
(exponential in $t$) lower bound for the second moment of $u$.  This theorem will be
proved in Section~\ref{S:Lbd}, where one can find more discussions regarding the
lower bounds of the second moment.

We believe that Theorem~\ref{T:Lbd} holds true for all $\rho>0$; the condition
requiring $\rho$ large enough so that $f_{\alpha,\rho}$ is positive is only a
technical assumption. Indeed, if we assume in addition that the initial data $\mu$ is given
by a bounded measurable function, we are able to prove the exponential lower
bound for all $\rho>0$.

\begin{theorem}\label{T:Lbd F-K}
  Assume that the initial condition is given by a bounded measurable function
  which is also bounded below away from zero. Then under Dalang's
  condition~\eqref{E:Dalang}, the second moment of the solution to the parabolic
  Anderson model satisfies the exponential lower bound for some $C,C'>0$,
  \begin{align*}
      \mathbb{E}\left[u(t,x)^2\right] \geq C e^{C't},
  \end{align*}
  when the driving noise satisfies $\rho>0$.
\end{theorem}

The heuristics of the above theorem go as follows. Assume for a moment that
$\mu\equiv 1$ so that we have the Feynman-Kac formula for the second moment:
\begin{align}\label{eq:F-K for moments}
  \mathbb{E}\left[u(t,x)^2\right] = \E\left[\exp\left\{\lambda^2\int_0^tf_{\alpha,\rho}(B_s, \widetilde{B}_s) \ud s \right\}\right].
\end{align}
In the above, $B$ and $\widetilde{B}$ are two independent Brownian motions on
$\mathbb{T}^d$ starting from $x$, and $\E$ is taking expectation with respect to
the Brownian motion. Since $\mathbb{T}^d$ is compact, the ergodic theorem
implies that the exponent in~\eqref{eq:F-K for moments} asymptotically and
almost surely becomes
\begin{align*}
  \int_0^tf_{\alpha,\rho}(B_s,\widetilde{B}_s) \ud s
  & \sim t \times \left(\frac{1}{(2\pi)^{2d}}\iint_{\mathbb{T}^{2d}}f_{\alpha,\rho}(x,y) \ud x  \ud y \right) \\
  & = \frac{\rho\, t}{(2\pi)^d}+ \frac{t}{(2\pi)^{2d}}\iint_{\mathbb{T}^{2d}}f_{\alpha}(x,y) \ud x  \ud y,\quad \mathrm{as}\ t\uparrow\infty.
\end{align*}
However, note that the integral on the right hand-side vanishes, by the
definition of $f_\alpha=f_{\alpha,0}$ in~\eqref{E:f}. Therefore, as
$t\uparrow\infty$,
\begin{align*}
  \mathbb{E}\left[u(t,x)^2\right]
  \geq \left[\exp\left\{\lambda^2\E\int_0^tf_{\alpha,\rho}(B_s, \widetilde{B}_s) \ud s \right\}\right]
  \sim \exp\left(\frac{\rho\lambda^2}{(2\pi)^d}t+o(t)\right).
\end{align*}
The above argument clearly suggests that the ergodicity of the Brownian motion
(due to compactness of the torus) is the main source that leads to the
exponential lower bound for the second moment. As a consequence, one always
observes intermittency in this situation; no phase transition takes place. We
expect that this is a general phenomenon when the state space is a compact
manifold without boundary. However, the question becomes more delicate when
$\rho=0$. In addition, Theorem~\ref{T:Lbd F-K} does not address the case when
the initial date is rough. These questions will be tackled in a future work.
\bigskip

Finally, the H\"{o}lder regularity of the solution $u$ is given as follows.
\begin{theorem}\label{T:Holder}
  Then if the noise correlation satisfies the strong Dalang
  condition~\eqref{E:Dalang} {for some $\alpha \in (0,d/2)$},
  the unique solution $u$ starting from a finite measure $\mu$ is
  $\beta_1$--H\"older continuous in time and $\beta_2$--H\"older continuous in
  space on $(0,\infty)\times\dom^d$ for all
  \begin{align*}
    \beta_1\in\left(0,{\frac{2\alpha+2-d}{4}}\right) \quad \text{and} \quad
    \beta_2\in\left(0,{\frac{2\alpha+2-d}{2}}\right).
  \end{align*}
\end{theorem}

This theorem will be proved in Section~\ref{S:Holder}. The proof of this theorem
follows similar arguments as the corresponding result for $\R^d$ (see
Theorem~1.8 of~\cite{chen.huang:19:comparison}) with more complexity introduced
by the fundamental solution.

\bigskip There has been a growing interest in the study of SPDE (and related
polymers models) on some ``exotic" spaces. For example, C. Cosco, I. Seroussi
and O. Zeitouni considered directed polymers on infinite graph
in~\cite{cosco.seroussi.ea:21:directed}. In a recent
work~\cite{baudoin.ouyang.ea:22:parabolic}, the authors studied the PAM on
Heisenberg groups. In addition, A. Mayorcas and H. Singh released a
preprint~\cite{mayorcas.singh:23:singular} recently studying singular SPDEs on
Homogeneous Lie Groups.

The construction of the colored noise on $\mathbb{T}^d$ presented in this paper
grows from a discussion between the second author and Prof. Elton Hsu during the
BIRS-CMO workshop ``Theoretical and Applied Stochastic Analysis'' in 2018.
Later, the second author was informed by Prof. Fabrice Baudoin that a fractional
noise is introduced in a similar spirit on general Riemannian manifolds
in~\cite{gelbaum:14:fractional}.

Our construction can be easily generalized to general (compact) Riemannian
manifolds, thereby encompassing a large class of spaces with rich geometric and
topological properties. This naturally raises questions: Which specific
geometric or topological properties might influence the behavior of the PAM, and
how might they introduce novel features to the model? The present work can be
regarded as an initial step in this direction. Notably, as elaborated in
Remark~\ref{R:novelty} above, our methodology does not necessarily rely on
Fourier analysis. The implementation of our techniques developed in this paper
in order to study the PAM on general manifolds will be given in a subsequent
work.

\bigskip The rest of the paper is organized as follows. In
Section~\ref{S:Prelim}, we prove some preliminary properties of the densities of
Brownian motion and Brownian bridge on torus. Section~\ref{S:Noise} is dedicated
to the construction of the colored noise on torus. The main technical step in
proving Theorem~\ref{T:Main} is prepared in Section~\ref{S:Kappa}. Then we prove
Theorem~\ref{T:Main} in Section~\ref{S:Main T}, and Theorems~\ref{T:Lbd}
and~\ref{T:Lbd F-K} in Section~\ref{S:Lbd}. Finally, the H\"{o}lder continuity
of the solution claimed in Theorem~\ref{T:Holder} is established in Section
\ref{S:Holder}. In order to improve the readability of the article, we deferred
the proof of Lemma~\ref{L:G-bdd} to Appendix~\ref{A:G-bdd}.

\section{Preliminary -- the fundamental solution and various properties}\label{S:Prelim}

In this section, we establish some properties of the fundamental solution.

\begin{lemma}\label{L:G-bdd}
  The fundamental solution $G(t,x)$ satisfies that
  \begin{align}\label{E:G-bdd}
    C_t^d \leq \frac{G(t,x)}{p(t,x)} \leq \left(2 C_t\right)^d, \quad \text{for all $(t,x)\in \R_+ \times \mathbb{T}^d$,}
  \end{align}
  where the constant $C_t$ as a function of $t$ can be expressed in the
  following equivalent expressions:
  \begin{alignat}{2}\label{E:Ct}
      C_t
      & \stackrel{(s) }{=} \sum_{n=-\infty}^{\infty} e^{-\frac{2n^2\pi^2}{t}}
      & \stackrel{(p) }{=} \prod_{n=1}^{\infty} \left(1- e^{-\frac{4n\pi^2}{t}}\right)\left(1+e^{-\frac{2(2n-1)\pi^2}{t}}\right)^2 \\
      & \stackrel{(s')}{=} \sqrt{\frac{t}{2\pi}} \sum_{n=-\infty}^{\infty} e^{-\frac{n^2 t}{2}}
      & \stackrel{(p')}{=} \sqrt{\frac{t}{2\pi}} \prod_{n=1}^{\infty} \left(1- e^{-nt}\right)\left(1+e^{-\frac{(2n-1)t}{2}}\right)^2.
  \end{alignat}
  Moreover, the following properties hold:
  \begin{enumerate}
    \item $C_t$ has the following asymptotic properties:
      \begin{align}\label{E:Ct-Asym}
        \lim_{t\to 0}  t \log\left(\log\left(C_t\right)\right) = -\pi^2
        \quad \text{and} \quad
        \lim_{t\to \infty} \frac{1}{t}\log\left(\log\left(C_t \sqrt{\frac{2\pi}{t}}\right)\right) = -\frac{1}{2};
      \end{align}
    \item $C_t$ satisfies the following bounds for all $t>0$:
      \begin{align}\label{E:Ct-bdd}
        \max\left(1,\sqrt{\frac{t}{2\pi}}\right) \le C_t \le 1 + \sqrt{\frac{t}{2\pi}};
      \end{align}
    \item For all $t>0$ and  $x\in \mathbb{T}^d$, it holds that
      \begin{align}\label{E:G-UnifBd}
        G(t,x) \le \left(\sqrt{\frac{2\pi}{t}}\:C_t\right)^d \le \left(1+\sqrt{\frac{2\pi}{t}}\right)^d;
      \end{align}
    \item For any $\epsilon>0$, there exists some constant $\Theta_{\epsilon,d}>0$
      such that for all $t\ge \epsilon$,
      \begin{align}\label{E:G-Inf}
        \sup_{x\in \mathbb{T}^d}\left\lvert G(t,x) - \left(\frac{1}{2\pi}\right)^d\right\rvert
        \le \Theta_{\epsilon,d}\: e^{-t/2}.
      \end{align}
  \end{enumerate}
\end{lemma}

The proof of Lemma~\ref{L:G-bdd} is given in Appendix~\ref{A:G-bdd}. The
estimate in~\eqref{E:G-Inf} can also be found in Theorem~2.15
of~\cite{baxter.brosamler:76:energy}. In Figure~\ref{F:Ct}, we plot some graphs
of $C_t$ as a function of $t$.

\begin{figure}[htpb]
  \centering
  \begin{tikzpicture}[scale=0.9, transform shape]
    \tikzset{>=latex}
      \begin{axis}[
        axis lines = left,
        ytick={0, 0.5, 1, 1.2, 1.4, 1.6},
        yticklabel style={
            /pgf/number format/fixed,
            /pgf/number format/precision=3
        },
        xmin=0, xmax= 16.8,
        ymin=0, ymax= 1.8,
        xlabel={$t$},
        ylabel={$C_t$},
        xlabel style={at=(current axis.right of origin), anchor=center, xshift=2.5em, yshift = +1.0em},
        ylabel style={at=(current axis.above origin), xshift=+1.5em, yshift = -2.0em, rotate=-90},
      ]
      \addplot[domain=0.001:16.0, blue, solid, very thick] table {theta-NearZero.csv};
      \addplot[domain=0.001:16.0, red, dashed, very thick] {sqrt(x/(2*pi))};
    \end{axis}
  \end{tikzpicture}
  \hfill
  \begin{tikzpicture}[scale=0.9, transform shape]
    \begin{axis}[
        axis lines = left,
        ytick={0,1,2,3,4,5,6},
        xmin=0, xmax= 105,
        ymin=0.8, ymax= 4.5,
        xlabel={$t$},
        ylabel={$C_t$},
        xlabel style={at=(current axis.right of origin), anchor=center, xshift=2.5em, yshift = +1.0em},
        ylabel style={at=(current axis.above origin), xshift=+2.0em, yshift = -2.0em, rotate=-90},
      ]
      \addplot[domain=0.001:100, blue, solid, very thick] table {theta-NearInf.csv};
      \addplot[domain=0.001:100, red, dashed, very thick] {sqrt(x/(2*pi))};
    \end{axis}
  \end{tikzpicture}
  \caption{Some plots (solid lines) of the function $C_t$ with plots (dashed lines) for
  $\sqrt{\frac{t}{2\pi}}$.}
  \label{F:Ct}
\end{figure}
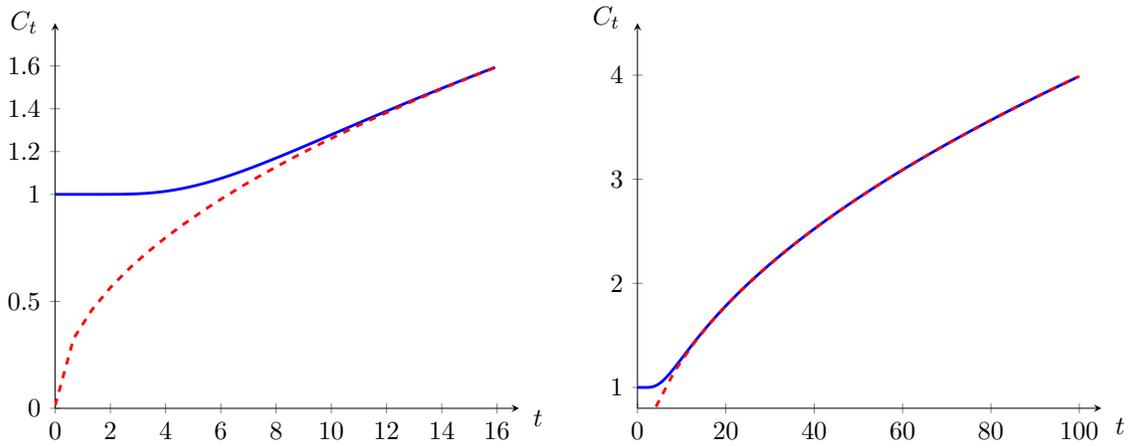

%
%
%
%

We will need to introduce the density of the \textit{pinned Brownian motion} or
\textit{Brownian bridge} (started at $x_0$ and terminating at $x$ at time $t$)
on $\mathbb{T}^d$:
\begin{align}\label{E:BMBridge-T}
  G_{t,x_0,x}(s,z)
  \coloneqq \frac{G(s,x_0,z)G(t-s,z,x)}{G(t,x_0,x)}, \quad
  \forall (x_0,x, z) \in \mathbb{T}^{3d},\, 0< s < t.
\end{align}
The corresponding transition density for the Brownian bridge on $\R^d$ is
\begin{align}\label{E:BMBridge-R}
  \begin{aligned}
    p_{t,x_0,x}(s,z)
    & \coloneqq \frac{p(s,x_0-z)p(t-s,z-x)}{p(t,x_0-x)}, \\
    & =p\left(\frac{s(t-s)}{t}, z-\left(x_0+\frac{s}{t}(x-x_0)\right)\right), \quad \forall (x_0,x, z) \in \R^{3d},\, 0< s < t.
  \end{aligned}
\end{align}

The following lemma states that when $t$ is ``large" and $s$ is small, the transition density of a Brownian bridge is comparable to that of a Brownian motion on torus.

\begin{lemma}\label{L:Large}
  Fix an arbitrary $\epsilon>0$. Suppose that $t\ge \epsilon$. Then for all
  $s\in [0,t/2]$ and $x_0,z\in \mathbb{T}^d$, the density for the Brownian
  bridge is comparable to a Gaussian on the torus,
  \begin{align}\label{E:Large}
    c_\epsilon^d\: G(s,x_0,z)\leq G_{t,x_0,x}(s,z) \leq C_{\epsilon}^d\: G(s,x_0,z),
  \end{align}
  where
  \begin{align}\label{E:C-Large}
    c_\epsilon \coloneqq \frac{\sqrt{\epsilon}}{2 \sqrt{\pi}+ \sqrt{2\epsilon}}
    \times \frac{1}{2 \sqrt{2}} e^{-\frac{\pi^2}{2\epsilon}}
    \quad \text{and} \quad
    C_\epsilon \coloneqq 2\left(1+ \sqrt{\frac{2\pi}{\epsilon}}\right) e^{\pi^2/\epsilon}.
  \end{align}
\end{lemma}
\begin{proof}
  We first prove the case when $d=1$. From~\eqref{E:G-bdd}, we see that
  \begin{align*}
    \frac{p_1\left(t,\myMod{x_0-x}\right)}{p_1\left(t-s,\myMod{z-x}\right)} \times \frac{C_{t-s}}{2C_t}
    \le \frac{G(t-s,z,x)}{G(t,x_0,x)} \le
    \frac{p_1\left(t,\myMod{x_0-x}\right)}{p_1\left(t-s,\myMod{z-x}\right)} \times \frac{2 C_{t-s}}{C_t}.
  \end{align*}
  Since $\myMod{x}\in [-\pi,\pi]$, we see that
  \begin{align*}
    \sqrt{\frac{t-s}{t}} e^{-\frac{\pi^2}{2t}}
    \le \frac{p_1\left(t,\myMod{x_0-x}\right)}{p_1\left(t-s,\myMod{z-x}\right)} \le
    \sqrt{\frac{t-s}{t}} e^{\frac{\pi^2}{2(t-s)}}.
  \end{align*}
  Since $t>\epsilon$ and $s\in (0,t/2)$, we see that $\epsilon/2\le t/2\le
  t-s\le t$. Hence,
  \begin{align*}
    \frac{1}{\sqrt{2}} e^{-\frac{\pi^2}{2\epsilon}}
    \le \frac{p_1\left(t,\myMod{x_0-x}\right)}{p_1\left(t-s,\myMod{z-x}\right)} \le
     e^{\pi^2/\epsilon}.
  \end{align*}
  From~\eqref{E:Ct-bdd}, we see that
  \begin{gather*}
    \frac{C_{t-s}}{C_t}
    \le \frac{1+ \sqrt{\frac{t-s}{2\pi}}}{\sqrt{\frac{t}{2\pi}}}
    \le \frac{\sqrt{2\pi}+ \sqrt{t}}{\sqrt{t}}\le 1+ \sqrt{\frac{2\pi}{\epsilon}} \qquad \text{and}\\
    \frac{C_{t-s}}{C_t}
    \ge \frac{\sqrt{\frac{t-s}{2\pi}}}{1+ \sqrt{\frac{t}{2\pi}}}
    \ge \frac{\sqrt{\frac{t}{4\pi}}}{1+ \sqrt{\frac{t}{2\pi}}}
    \ge \frac{\sqrt{\frac{\epsilon}{4\pi}}}{1+ \sqrt{\frac{\epsilon}{2\pi}}}
    = \frac{\sqrt{\epsilon}}{2 \sqrt{\pi}+ \sqrt{2\epsilon}}.
  \end{gather*}
  Combining the above inequalities proves the case of $d=1$ with the constants
  given in the statement of the lemma. As for case for $d\ge 2$, one only needs
  to raise above constants $c_\epsilon$ and $C_\epsilon$ to the power of $d$.
\end{proof}

The next lemma is our formal statement that when $t$ is small, the Brownian Bridge on torus can be compared to that on $\mathbb{R}^d$:
\begin{lemma}\label{L:BMBridge}
  There exists a universal constant $C>0$ such that for all $t>s>0$ and
  $x,x_0,z\in \mathbb{T}^d$ with $z-x_0\in \mathbb{T}^d$, it holds that
  \begin{align*}
    G_{t,x_0,x}(s,z)
    \le C \left(1+\sqrt{t}\right)^d\sum_{k\in \Pi^d} p_{t,x_0,x+k}(s,z),
  \end{align*}
  where
  \begin{align}\label{E:Pi}
    \Pi \coloneqq \left\{-2\pi, 0, 2\pi\right\}.
  \end{align}
  Note that $p_{t,x_0,x}(s,z)$ is a function on $\mathbb{R}^d$ (as apposed to a function on torus). The condition $z-x_0\in \mathbb{T}^d$ in the above is simply a concise way to state that, when evaluated by $p_{t,x_0,x}(s,\cdot)$, $z$ takes values in $[x^1_0-\pi,x^1_0+\pi]\times\cdots\times[x^d_0-\pi,x^d_0+\pi]\subset\mathbb{R}^d$ for $x_0=(x^1_0,\dots,x^d_0).$
\end{lemma}
\begin{proof}
  It suffices to prove the case $d=1$. Fix arbitrary $t>s>0$. Let $x,x_0,z\in
  \mathbb{T}$ with $z-x_0\in \mathbb{T}$. From~\eqref{E:G-bdd}, we see that
  \begin{align*}
    G_{t,x_0,x}(s,z)
    = \frac{G(s,z-x_0)G\left(t-s, \myMod{x-z}\right)}{G(t,\myMod{x-x_0})}
    \le 4\frac{C_s C_{t-s}}{C_t} \frac{p(s,z-x_0) p\left(t-s,\myMod{x-z}\right)}{p\left(t,\myMod{x-x_0}\right)},
  \end{align*}
  where we have used the fact that $z-x_0\in \mathbb{T}$ in the equality. By the
  properties of $C_t$ given in Lemma~\ref{L:G-bdd}, we have that
  \begin{align*}
    \frac{C_s C_{t-s}}{C_t} \le \frac{C_t C_{t}}{C_t} = C_t \le 1+\sqrt{\frac{t}{2\pi}}.
  \end{align*}
  Hence, for some universal constant $C>0$,
  \begin{align*}
    G_{t,x_0,x}(s,z)
    & \le C\left(1+\sqrt{t}\right) \frac{p(s,z-x_0) p\left(t-s,\myMod{x-z}\right)}{p\left(t,\myMod{x-x_0}\right)}.
  \end{align*}

  In order to determine the values of the modulo functions, by taking into
  account of the fact that $x,x_0,z\in [-\pi,\pi]$, we consider the following
  three cases:

   \bigskip\noindent\textbf{Case I:~} $x-x_0\in (\pi,2\pi]$. In this case, as
   illustrated in Figure~\ref{F:CaseI}, we have
   \begin{align*}
     \myMod{x-x_0} = x-x_0 - 2\pi \quad \text{and} \quad
     \myMod{x-z} \in \left\{x-z+k:\: k=-2\pi,0\right\}.
   \end{align*}
   Hence,
   \begin{align}\label{E_:RatioHeat}
     \frac{p(s,z-x_0) p\left(t-s,\myMod{x-z}\right)}{p\left(t,\myMod{x-x_0}\right)}
     & \le \frac{p(s,z-x_0)}{p\left(t,x-x_0-2\pi\right)} \sum_{k\in \{-2\pi,0\}} p\left(t-s,x-z+k\right) \notag\\
     & =  \sum_{k \in \{-2\pi,0\}}\frac{p(t,x-x_0+k)}{p\left(t,x-x_0-2\pi\right)} \times p_{t,x_0,x+k}(s,z).
   \end{align}
   One can routinely check that in this case, $|x-x_0+k|\ge |x-x_0-2\pi|$
   for $k\in \{0,-2\pi\}$. Hence, the ratio of two heat kernels
   in~\eqref{E_:RatioHeat} is bounded by one. Therefore,
   \begin{align*}
    \frac{p(s,z-x_0) p\left(t-s,\myMod{x-z}\right)}{p\left(t,\myMod{x-x_0}\right)}
     & \le  \sum_{k\in \{-2\pi,0\}} p_{t,x_0,x+k}(s,z).
   \end{align*}

   \begin{figure}[htpb]
    \centering
    \begin{tikzpicture}[x=4em, y=4em]
    \tikzset{>=latex}
      \def\radius{1}
      \draw[thick] (0,-\radius) arc (-90:269:\radius);
      \draw[thick] (0,+1) --++ (0,0.07) node[above] {$0$} --++(0,-0.14);
      \draw[thick] (-0.03,-1) --++ (0,+0.13) --++(-0.05,0);
      \draw[thick] (-0.03,-1) --++ (0,-0.13) --++(-0.05,0);
      \draw[thick] (+0.00,-1) --++ (0,+0.13) --++(+0.05,0);
      \draw[thick] (+0.00,-1) --++ (0,-0.13) --++(+0.05,0);
      \node at (-0.25,-1.15) {$-\pi$};
      \node at (+0.12,-0.85) {$\pi$};

      \begin{scope}[rotate = -60]
        \draw[thick] (0,-1) --++ (0,-0.10) --++(0,0.20) node [left, yshift = -6pt, xshift = -6pt, blue] {$x_0$};
        \draw[thick] (0,+1) --++ (0,-0.10) --++(0,0.20);
        \draw[dotted] (0,-1) -- (0,1);
      \end{scope}
      \begin{scope}[rotate = +60]
        \def\tick{0.35}
        \draw[thick] (0,-1) --++ (0,+0.10) --++(0,-0.10-\tick) node [right, yshift = -2pt, xshift = 0pt] {$x$};
        \draw[thick] (0,+1) --++ (0,-0.10) --++(0,0.20);
        \draw[dotted] (0,-1) -- (0,1);
      \end{scope}

      \begin{scope}[rotate = +150]
        \draw[line width = 2.3pt,blue, opacity =0.5] (1.03*\radius,0) arc (0:118:1.03*\radius);
      \end{scope}

      \begin{scope}[rotate = +30]
        \draw[line width = 2.3pt,red, opacity =0.5] (0.97*\radius,0) arc (0:238:0.97*\radius);
      \end{scope}

      \begin{scope}[rotate = -30]
        \def\tick{0.18}
        \draw[thick] (0,1) --++ (0,-0.10) node[left, yshift = -5pt, xshift = +7pt, red] {$z_1$} --++(0,+0.10+\tick);
      \end{scope}
      \begin{scope}[rotate = 150]
        \def\tick{0.35}
        \draw[thick] (0,1) --++ (0,-0.10) node[above, yshift = -2pt, xshift = +5pt, red] {$z_2$} --++(0,+0.07+\tick);
      \end{scope}

      \begin{scope}[rotate = -30]
        \draw[<->,dashed,thick] (1.15*\radius,0) arc (0:90:1.15*\radius);
      \end{scope}
      \node at (2.2,0.3) {$\myMod{x-z_1} = x-z_1$};

      \begin{scope}[rotate = -120]
        \draw[<->,dashed,thick] (1.28*\radius,0) arc (0:90:1.28*\radius);
      \end{scope}

      \node at (2.0,-1.26) {$\myMod{x-z_2} = x-z_2-2\pi$};
    \end{tikzpicture}

    \caption{Illustration of Case I in the proof of Lemma~\ref{L:BMBridge}.
    Since $x-x_0\in (\pi,2\pi]$ and $x,x_0\in [-\pi,\pi]$, we see that $x$ has
  to situate in $(0,\pi]$. The possible positions of $x_0$ are given in the blue
arc. Since $z-x_0\in [-\pi,\pi]$, the possible positions of $z$ are highlighted
in the red arc. When computing $\myMod{x-z}$, we have two cases illustrated by
$z_1$ and $z_2$.}

    \label{F:CaseI}
  \end{figure}
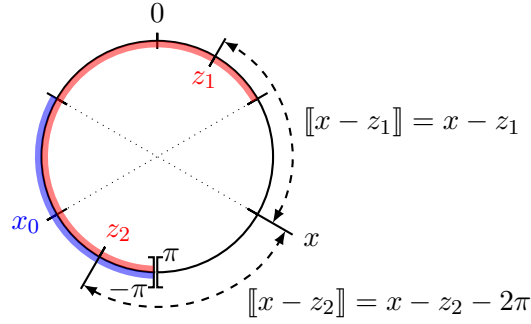

   \bigskip\noindent\textbf{Case II:~} $x-x_0\in [-\pi,\pi]$. In this case, we
   have that
   \begin{align*}
     \myMod{x-x_0} = x-x_0 \quad \text{and} \quad
     \myMod{x-z} \in \left\{x-z +k:\: k=-2\pi,0,2\pi\right\}.
   \end{align*}
   By the same arguments as in Case~I, one can show that the ratio of two heat
   kernels in~\eqref{E_:RatioHeat} is bounded by one and hence,
   \begin{align*}
     \frac{p(s,z-x_0) p\left(t-s,\myMod{x-z}\right)}{p\left(t,\myMod{x-x_0}\right)}
     & \le  \sum_{k\in \{-2\pi,0,2\pi\}} p_{t,x_0,x+k}(s,z).
   \end{align*}

   \bigskip\noindent\textbf{Case III:~} $x-x_0\in [-2\pi,-\pi)$. This is the
   symmetric case of Case I. In this case,
   \begin{align*}
     \myMod{x-x_0} = x-x_0+2\pi \quad \text{and} \quad
     \myMod{x_0-x+z} \in \left\{x-x_0+z+k:\: k=0,2\pi\right\}.
   \end{align*}
   By similar arguments, we see that
   \begin{align*}
     \frac{p(s,z) p\left(t-s,\myMod{x_0+z-x}\right)}{p\left(t,\myMod{x_0-x}\right)}
     & \le  \sum_{k \in \{0,2\pi\}} p_{t,x_0,x+k}(s,z).
   \end{align*}

   Combining the above three cases proves Lemma~\ref{L:BMBridge} for the case
   $d=1$. The generalization to the case for $d\ge 2$ is straightforward and we
   omit the details here.
\end{proof}

\section{Colored noise on flat Torus}\label{S:Noise}

In this section, we construct a family of intrinsic Gaussian noises on the flat
torus $\mathbb{T}^d$ that we call \textit{colored noise} on $\mathbb{T}^d$ that
is white in time. Recall that the eigenvectors of the Laplace $\triangle$ are
given by $\exp{(ik\cdot x)}, k\in\mz^d.$ For any $\varphi\in L^2(\mt^d)$, there
is a unique decomposition
\begin{align*}
  \varphi(x) = (2\pi)^{-d/2}\sum_{k\in\mz^d}a_k\,e^{ik\cdot x},
\end{align*}
where $a_k$ are the Fourier coefficients of $\varphi$:
\begin{align*}
  a_k = \mathcal{F}(\varphi)(k) \coloneqq (2\pi)^{-d/2}\int_{\mathbb{T}^d} \varphi(x) e^{-i k\cdot x}\ud x,
  \quad \text{for all $k\in \mathbb{Z}^d$}.
\end{align*}
In particular, $a_0=(2\pi)^{-d/2}\int_{\mathbb{T}^d}\varphi(x) \ud x$.

We introduce a family of Gaussian noises $\dot{W}$ on $\mt^d$ with parameters
$\alpha$, $\rho\ge 0$ as follows. Let $(\Omega, \mathcal{F},\mathbb{P})$ be a
complete probability space such that for any $\varphi(x)$ and $\psi(x)$ on
$\mt^d$ and $t,s>0$, both $\dot{W}\left(1_{[0,t]}\varphi\right)$ and
$\dot{W}\left(1_{[0,s]}\psi\right)$ are centered Gaussian random variables with
covariance given by
\begin{align}\label{E:NoiseCov}
  \begin{gathered}
    \me \left(\dot{W}\left(1_{[0,t]}\varphi\right)\dot{W}\left(1_{[0,s]}\psi\right)\right)
    = (s\wedge t) \langle\phi,\psi\rangle_{\alpha,\rho} \quad \text{with}\\
    \langle\phi,\psi\rangle_{\alpha,\rho}
    \coloneqq \rho a_0\bar{b}_0+\sum_{k\in\mz^d_*} \frac{a_k\bar{b}_k}{|k|^{2\alpha}} \quad \text{and} \quad \mz_*^d \coloneqq \mz^d \setminus \{0\},
  \end{gathered}
\end{align}
where $a_k$'s and $b_k$'s are the Fourier coefficients of $\varphi$ and $\psi$,
respectively.

For $\rho>0$, let $\mathcal{H}^{\alpha,\rho}$ be the completion of
$L^2(\mathbb{T}^d)$ under $\langle\cdot,\cdot\rangle_{\alpha,\rho}$. Then,
$(\Omega, \mathcal{H}^{\alpha,\rho},\mathbb{P})$ gives an abstract Wiener space.
When $\rho=0$, some special care is needed in order to identify a suitable
Hilbert space $\mathcal{H}^{\alpha,0}$. Let $L^2_0$ be the space of $L^2$
functions on $\mathbb{T}^d$ such that $a_0=0$. Denote by
$\mathcal{H}^{\alpha}_0$ the completion of $L^2_0$ under
$\langle\cdot,\cdot\rangle_{\alpha,\rho}$. One could have set
$\mathcal{H}^{\alpha,0}=\mathcal{H}^{\alpha}_0$. However, when solving SPDEs on
torus, it is desirable to consider Wiener integrals
$\dot{W}\left(1_{[0,t]}\varphi\right)$ where $\varphi$ is a function on the
torus such that $a_0=(2\pi)^{-d/2}\int_{\mathbb{T}^d}\varphi(x) \ud x \not=0$.
For this purpose, consider $\mathcal{H}^{\alpha}_0+\mathbb{R} \coloneqq
\{\varphi+c: \varphi\in\mathcal{H}^{\alpha}_0, \textnormal{and}\
c\in\mathbb{R}\}$. We can identify $\mathcal{H}^{\alpha}_0+\mathbb{R}$ with
$\mathcal{H}^{\alpha}_0$ through the equivalence relation $\sim$, in which
$\varphi\sim\psi$ if $\varphi-\psi$ is a constant. Finally, we set
\begin{align*}
  \mathcal{H}^{\alpha,0}=(\mathcal{H}^{\alpha}_0+\mathbb{R})/ \sim.
\end{align*}
With this construction, it is clear that
$\dot{W}\left(1_{[0,t]}\varphi\right)=\dot{W}\left(1_{[0,t]}(\varphi+c)\right)$
for any $\varphi\in\mathcal{H}^{\alpha,0}$ and $c\in\mathbb{R}$. Throughout the
rest of our discussion, we will also adopt the short-hand $\mathcal{H}^\alpha$
for $\mathcal{H}^{\alpha,0}$.

\begin{remark}
  It is clear from~\eqref{E:NoiseCov} that $L^2(\mathbb{T}^d)\subset
  \mathcal{H}^{\alpha,\rho}\subset\mathcal{H}^{\beta,\rho}$ for
  $0\leq\alpha<\beta$. Moreover, the colored noise becomes the white noise on
  torus when $\rho=1$ and $\alpha=0$.
\end{remark}

Recall that $f_{\alpha,\rho}$ and $f_\alpha$ are defined in~\eqref{E:f} and
\eqref{E:f_alpha} respectively.

\begin{lemma}\label{L:Cov}
  Fix arbitrary $\alpha>0$ and $\rho\geq0$. Let $\Theta = \Theta_{1,d}$ be the
  constant given in~\eqref{E:G-Inf}. The following statements hold:
  \begin{enumerate}[(i)]
    \item $\int_{\mathbb{T}^d}f_\alpha(x)\ud x = 0$;
    \item $f_\alpha(x)$ assume both positive and negative values;
    \item $f_{\alpha,\rho}(x)$ is bounded from below:
      \begin{align*}
        f_{\alpha,\rho}(x)\geq (2\pi)^{-d/2} \left(\rho -\frac{1}{\Gamma(\alpha+1) (2\pi)^{d/2}} - (2\pi)^{d/2} 2^\alpha\Theta\right),
        \quad \textnormal{for all $x\in\mathbb{T}^d$;}
      \end{align*}
    \item $f_{\alpha,\rho}(\cdot)$ is nonnegative when $\rho\ge
      (2\pi)^{-d/2}\Gamma(\alpha+1)^{-1} + (2\pi)^{d/2}2^\alpha\Theta$;
    \item The Fourier coefficients for $f_{\alpha,\rho}$ are given by
      \begin{align}\label{E:theta}
        \theta_n \coloneqq \frac{1}{(2\pi)^{d/2}}\int_{\mathbb{T}^d}f_{\alpha,\rho}(x)e^{-i n\cdot x}\ud x =
        \begin{dcases}
          \, \frac{\rho}{(2\pi)^{d/2}}        & \text{if $n = 0$} \\
          \frac{1}{|n|^{2\alpha}(2\pi)^{d/2}} & \text{if $n\in \mathbb{Z}^d_*$}
        \end{dcases};
      \end{align}
    \item For any $\varphi$ and $\psi\in \mathcal{H}^{\alpha,\rho}$, it holds
      that
      \begin{align}\label{E:NoiseCov2}
        \InPrd{\varphi,\psi}_{\alpha,\rho} = \iint_{\mathbb{T}^{2d}} \varphi(x) f_{\alpha,\rho}(x,y) \psi(y) \ud x\ud y.
      \end{align}
      As a consequence, the noise introduced in~\eqref{E:NoiseCov} can be
      equivalently expressed by
      \begin{align*}
        \me \left(\dot{W}(1_{[0,t]}\varphi)\dot{W}(1_{[0,s]}\psi)\right)
         =\left(t\wedge s\right)\iint_{\mathbb{T}^{2d}}\varphi(x)f_{\alpha,\rho}(x,y)\psi(y) \ud x  \ud y .
      \end{align*}
  \end{enumerate}
\end{lemma}
Note that the lower bounds in parts (iii) and (iv) are not optimal.
\begin{proof}
  By writing
  \begin{align}\label{E:I_1-I_2}
    f_\alpha(x) = \frac{1}{\Gamma(\alpha)}\int_0^\infty t^{\alpha-1}\left(G(t,x)-(2\pi)^{-d}\right) \ud t =\int_0^1+\int_1^\infty=I_1+I_2,
  \end{align}
  thanks to the heat kernel estimate in~\eqref{E:G-Inf}, one can apply the
  dominated convergence theorem to switch the $\ud x$- and the $\ud
  t$-integrals. This yields part (i). Part (ii) is an immediate consequence of
  part (i). As for part (iii), from~\eqref{E:I_1-I_2}, we see that
  \begin{align*}
    I_1
    =   \frac{1}{\Gamma(\alpha)}\int_0^1 t^{\alpha-1} \left(G(t,x)-\frac{1}{(2\pi)^d}\right) \ud t
    \ge -\frac{1}{\Gamma(\alpha)(2\pi)^{d}}\int_0^1 t^{\alpha-1} \ud t
    =   -\frac{1}{\Gamma(\alpha+1)(2\pi)^d},
  \end{align*}
  where the inequality is due to the positivity of the heat kernel. As for
  $I_2$, for the constants $\Theta = \Theta_{1,d}$ given in~\eqref{E:G-Inf} and
  $\gamma = 1/2$,
  \begin{align*}
    I_2
    & = \frac{1}{\Gamma(\alpha)}\int_1^\infty t^{\alpha-1}e^{-\gamma t} e^{\gamma t}\left(G(t,x)-\frac{1}{(2\pi)^d}\right) \ud t \\
    & \ge - \frac{\Theta}{\Gamma(\alpha)}\int_0^\infty t^{\alpha-1}e^{-\gamma t} \ud t
      \ge - \Theta \gamma^{-\alpha}.
  \end{align*}
  This proves part (iii). Part (iv) is a direct consequence of part (iii).

  As for part (v), by part (i), we see that $\theta_0 = (2\pi)^{-d/2}\rho$. For $n\ne 0$,
  using the heat kernel represented via the eigenvectors of the Laplace
  operator, namely,
  \begin{align}\label{E:HeatSpec}
    G(t,x,y)=(2\pi)^{-d}\sum_{k\in\mathbb{Z}^d}{\color{blue}e^{-\frac{|k|^2}{2}t}}e^{ik\cdot x}e^{-ik\cdot y},
  \end{align}
  by an application of Fubini's theorem, we see that
  \begin{align*}
    \theta_n = \frac{1}{\Gamma(\alpha)}\int_0^\infty t^{\alpha-1} \left((2\pi)^{-d}e^{-|n|^2t} \int_{\mathbb{T}^d}(2\pi)^{-d/2}\ud x - 0\right)
             = \frac{1}{|n|^{2\alpha}(2\pi)^{d/2}}.
  \end{align*}
  This proves part (v). Finally, denote the double integral
  in~\eqref{E:NoiseCov2} by $I$. Then by the Plancherel theorem, we see that
  \begin{align*}
    I & = \InPrd{\varphi, \overline{f_{\alpha,\rho}*\psi}}_{\mathbb{T}^{d}}
        = \sum_{k\in \mathbb{Z}^d} \mathcal{F}(\varphi)(k) \overline{\mathcal{F}(f_{\alpha,\rho}*\psi)(k)} \\
      & = \sum_{k\in \mathbb{Z}^d} \mathcal{F}(\varphi)(k) \left[(2\pi)^{d/2}\times \overline{\mathcal{F}(f_{\alpha,\rho})(k)}\times \overline{\mathcal{F}(\psi)(k)}\right]\\
      & = (2\pi)^{d/2} \sum_{n\in \mathbb{Z}^d} a_n \overline{b}_n \theta_n,
  \end{align*}
  where $a_n$ and $b_n$ are Fourier coefficients of $\varphi$ and $\psi$,
  respectively, and $\theta_n\ge 0$ are given in~\eqref{E:theta}. This completes
  the proof of Lemma~\ref{L:Cov}.
\end{proof}

Now we introduce some temporal functions that are determined by the noise and
will appear in the Picard iterations in the proof of the main
result---Theorem~\ref{T:Main}. Define
\begin{align}\label{E:k1}
  k_1(s) = k_1(s;\alpha,\rho) \coloneqq \sum_{k\in\mathbb{Z}^d}\mathcal{F}(f_{\alpha,\rho})(k)e^{-s|k|^2}.
\end{align}
The next lemma gives some estimates on $k_1(s)$.

\begin{lemma}\label{L:k1}
  \begin{enumerate}
    \item For any $\beta>\max\left(-\alpha+d/2,0\right)$ and $s>0$, it holds
      that
      \begin{align}\label{E:k1-Ubd}
        k_1(s;\alpha,\rho) \le \frac{\rho}{(2\pi)^{d /2}}+ C_{\alpha,\beta,d} \: s^{-\beta}
        \quad \text{with}\quad
        C_{\alpha,\beta,d} \coloneqq \frac{\beta^\beta e^{-\beta}}{(2\pi)^{d/2}}\sum_{k\in\mathbb{Z}^d_*}\frac{1}{|k|^{2(\alpha+\beta)}} <\infty.
      \end{align}
    \item Under Dalang's condition~\eqref{E:Dalang}, namely, $2(\alpha+1)>d$,
      for all $t\ge 0$ and $\gamma>0$, we have that
      \begin{align}\label{E:k1-Int}
        \int_0^tk_1(s;\alpha,\rho) \ud s
        \le  \frac{\rho\, t}{(2\pi)^{d/2}} + C_{\alpha,d},
        \quad \text{where} \quad C_{\alpha,d}\coloneqq \frac{1}{(2\pi)^{d/2}}\sum_{k\in\mathbb{Z}^d_*}\frac{1}{|k|^{2\alpha+2}} <\infty,
      \end{align}
      and
      \begin{align}\label{E:k1-Laplace}
        \int_0^\infty e^{-\gamma s}  k_1(s;\alpha,\rho) \ud s
        = \frac{\rho}{(2\pi)^{d/2}} \frac{1}{\gamma} + \frac{1}{(2\pi)^{d/2}}\sum_{k\in\mathbb{Z}^d_*}\frac{1}{|k|^{2\alpha} \left(|k|^2+\gamma\right)} <\infty.
      \end{align}
  \end{enumerate}
\end{lemma}
\begin{proof}
  (1) From~\eqref{E:theta}, we see that
  \begin{align}\label{E_:k1}
    k_1(s) & = \frac{\rho}{(2\pi)^{d/2}} + \frac{1}{(2\pi)^{d/2}}\sum_{k\in\mathbb{Z}^d_*}\frac{1}{|k|^{2\alpha}}e^{-s|k|^2} \\
           & = \frac{\rho}{(2\pi)^{d/2}} + \frac{1}{(2\pi)^{d/2}}\sum_{k\in\mathbb{Z}^d_*}\frac{1}{|k|^{2(\alpha+\beta)}}e^{-s|k|^2 +2\beta \log\left(|k|\right)}. \notag
  \end{align}
  Denote $g_{s,\beta}(r)\coloneqq -s r^2 + 2\beta \log (r)$. By solving
  $g_{s,\beta}'(r)=0$, we find that $g_{s,\beta}(r)$ is maximized at $r_0 =
  \sqrt{\beta/s}$, and the maximum value is equal to $g_{s,\beta}(r_0) =
  \beta^\beta s^{-\beta} e^{-\beta}$. The condition that $\beta>- \alpha + d/2$
  implies that the remaining summation in $k$ is finite.

  (2) From the expression for $k_1(s)$ in~\eqref{E_:k1}, one easily sees that
  $k_1(s)$ is non-increasing and nonnegative. Moreover, $k_1(s)$ is integrable
  at $s=0$ because from~\eqref{E_:k1},
  \begin{align*}
    0 \le \int_0^tk_1(s) \ud s
      \le & \frac{\rho\, t}{(2\pi)^{d/2}} + \frac{1}{(2\pi)^{d/2}}\sum_{k\in\mathbb{Z}^d_*}\frac{1}{|k|^{2\alpha}}\int_0^\infty \ud s \, e^{-s|k|^2} \\
       =  & \frac{\rho\, t}{(2\pi)^{d/2}} + \frac{1}{(2\pi)^{d/2}}\sum_{k\in\mathbb{Z}^d_*}\frac{1}{|k|^{2\alpha+2}} <\infty,
  \end{align*}
  where the last inequality is due to Dalang's condition~\eqref{E:Dalang}. This
  proves~\eqref{E:k1-Int}. The equality in ~\eqref{E:k1-Laplace} can be proved
  in the same way. This proves Lemma~\ref{L:k1}.
\end{proof}

Let $f_\alpha^*$ denote the Riesz kernel on $\R^d$ and $\widehat{f}_\alpha^*$
be its Fourier transform, i.e.,
\begin{align}\label{E:Riesz}
  f_\alpha^*(x) \coloneqq |x|^{-d+2\alpha} \quad \text{and} \quad
  \widehat{f}_\alpha^*(\xi) = c_{d,\alpha}|\xi|^{-2\alpha}, \quad \text{for all $x$ and $\xi\in \R^d$.}
\end{align}
Similar to~\eqref{E:k1}, define
\begin{align}\label{E:k2}
  k_2(s;\alpha)
  = k_2(s)
  \coloneqq \int_{\mr^d}\widehat{f}_\alpha^* (\xi)\exp\left(-\frac{s|\xi|^2}{2}\right) \ud\xi
  = C_{d,\alpha}\, s^{\alpha-d/2},
\end{align}
where the last equality is an easy exercise (see, e.g., Example~1.2
of~\cite{chen.kim:19:nonlinear}). Hence, Dalang's condition~\eqref{E:Dalang}
ensures the integrability of $k_2(s)$ at $s=0$ and
\begin{align}\label{E:k2-Laplace}
  \int_0^\infty e^{-\gamma s} k_2(s;\alpha) \ud s
  = C_{d,\alpha}' \gamma ^{-\alpha +\frac{d}{2}-1} \quad \text{for all $\gamma>0$}.
\end{align}
Define $h_0(t; \alpha,\rho)\coloneqq 1$ and inductively for $n \ge 1$,
\begin{align}\label{E:hn}
  h_{n+1}(t;\alpha,\rho) = h_{n+1}(t) \coloneqq \int_0^t h_n(t-s)\big(k_1(s)+k_2(s)+1\big) \ud s.
\end{align}

The following lemma can be proved in the same way as Lemma~2.6
of~\cite{chen.kim:19:nonlinear} with $k(s)$ there replaced by $k_1(s)+k_2(s)+1$.

\begin{lemma}\label{L:h-inc}
  All functions $h_n(\cdot)$, $n\geq 1$, defined in~\eqref{E:hn} are
  nondecreasing on $\R_+$.
\end{lemma}

For any $\lambda\ne 0$, define
\begin{align}\label{E:Ht}
  H_\lambda(t) \coloneqq \sum_{n=0}^{\infty} \lambda^{2n} h_n(t).
\end{align}

\begin{lemma}\label{L:LogH}
  Suppose that Dalang's condition~\eqref{E:Dalang} holds. For all $\lambda\ne
  0$, it holds that
  \begin{align}\label{E:LogH}
    \limsup_{t\to\infty} \frac{1}{t}\log H_\lambda(t)
    \le \gamma_0(\lambda) \coloneqq \inf \left\{\gamma: \: \lambda^2 \Theta_\gamma <1\right\}
     <  \infty,
  \end{align}
  where
  \begin{align}\label{E:Theta_gamma}
    \Theta_\gamma \coloneqq
      \frac{\rho}{(2\pi)^{d/2}} \frac{1}{\gamma}
      + \frac{1}{(2\pi)^{d/2}}\sum_{k\in\mathbb{Z}^d_*}\frac{1}{|k|^{2\alpha} \left(|k|^2+\gamma\right)}
      + C_{d,\alpha} \gamma ^{-(\alpha+1- d/2)} + \gamma^{-1} < \infty.
  \end{align}
  Moreover, when $\lambda$ is large enough,
  \begin{align}\label{E:LogH-gamma}
    \gamma_0(\lambda) \lesssim \lambda^{\max\left(\frac{4}{2(1+\alpha) -d},\,2\right)}.
  \end{align}
\end{lemma}

One may check Lemma~2.5 of~\cite{chen.kim:19:nonlinear} or Lemma~A.1
of~\cite{balan.chen:18:parabolic} (see also ~\cite{foondun.khoshnevisan:13:on})
for similar accounts.

\begin{proof}
  For any $\gamma>0$,
  \begin{align}\label{E:H-Laplace}
    \int_0^\infty e^{-\gamma t} H_{\lambda}(t) \ud t
    = & \sum_{n=0}^{\infty} \lambda^{2n} \left[\int_0^\infty e^{-\gamma t} \left(k_1(s)+k_2(s) +1\right) \ud t\right]^n
    =   \sum_{n=0}^{\infty} \lambda^{2n} \Theta_\gamma^n,
  \end{align}
  where from~\eqref{E:k1-Laplace} and~\eqref{E:k2-Laplace}, we obtain the
  expression of $\Theta_\gamma$ in~\eqref{E:Theta_gamma}. Thanks to Dalang's
  condition~\eqref{E:Dalang}, $\Theta_\gamma$ is finite. Because
  $\Theta_\gamma\downarrow 0$ as $\gamma\uparrow \infty$, we see that
  $\gamma_0(\lambda)$ in ~\eqref{E:LogH} is well-defined and is finite.

  It is not easy to compute $\gamma_0(\lambda)$ since the dependence on $\gamma$
  in $\Theta_\gamma$ is implicit. Instead, we can obtain an estimate of
  $\gamma_0$ using~\eqref{E:k1-Ubd}. Let $k_1'(s)$ be the upper bound of
  $k_1(s)$ given in\eqref{E:k1-Ubd}, namely,
  \begin{align*}
    k_1^\dagger(s) = C_{\rho,\alpha,\beta,d} \left(1+s^{-\beta}\right),
    \quad \text{for}\quad \beta\in \left(\max\left(-\alpha + d/2,0\right),1\right).
  \end{align*}
  Accordingly, we define $h_n^\dagger(t)$, $H^\dagger_\lambda(t)$,
  $\Theta_\gamma^\dagger$, and $\gamma_0^\dagger(\lambda)$. It is clear that
  \begin{align*}
    H_\lambda(t) \le H_\lambda^\dagger(t) \quad \text{for all $t\ge 0$} \quad \text{and}\quad
    \gamma_0(\lambda) \le \gamma_0^\dagger(\lambda).
  \end{align*}
  Now, $\Theta_\gamma^\dagger$ has a more explicit expression:
  \begin{align*}
    \Theta_\gamma^\dagger
    & =
      \frac{\rho}{(2\pi)^{d/2}} \frac{1}{\gamma}
      + C_{\rho,\alpha,\beta,d} \left(\gamma^{-1}+\Gamma\left(1-\beta\right)\gamma^{-1 + \beta}\right)
      + C_{d,\alpha} \gamma ^{-(\alpha+1- d/2)} + \gamma^{-1} \\
    & = C_{\rho,\alpha,\beta,d}' \left(\frac{1}{\gamma} + \frac{1}{\gamma^{\alpha+1-d/2}} + \frac{1}{\gamma^{1-\beta}}\right)
      \asymp \frac{1}{\gamma^{1-\beta}}, \quad  \text{as $\gamma\to\infty$,}
  \end{align*}
  where the asymptotic form is due to the fact that
  \begin{align*}
    \max\left(-\alpha + \frac{d}{2},0\right)  < \beta < 1 \quad \Longleftrightarrow  \quad
    0 < 1 - \beta < \min\left( \alpha + 1 -\frac{d}{2},1\right).
  \end{align*}
  Therefore, when $\lambda$ is large enough we have
  \begin{align*}
    \gamma_0^\dagger \left(\lambda\right) = \inf \left\{\gamma:\, \lambda^2 \Theta_\gamma^\dagger < 1\right\}  \lesssim \lambda^{\frac{2}{1-\beta}}.
  \end{align*}
  Finally, replacing $\beta$ by $\max\left(-\alpha+d/2,0\right)$ in the above
  upper bound completes the proof of Lemma~\ref{L:LogH}.
\end{proof}

\section{Resolvent kernel function \texorpdfstring{$\mathcal{K}_\lambda$}{}}\label{S:Kappa}

Let us first introduce some functions:

\begin{definition}\label{D:Trangle}
  For $h, w: \R_+ \times \mathbb{T}^{4d} \to \R$, define the space-time
  convolution operator ``$\triangleright$'' by
  \begin{align*}
    \left(h \triangleright w\right)(t,x_0,x,x_0',x')
    \coloneqq \int_{0}^{t}\ud s \iint_{\mathbb{T}^{2d}} \ud z\ud z' h\left(t-s,z,x,z',x'\right) w\left(s,x_0,z,x_0',z'\right)f_{\alpha,\rho}(z,z').
  \end{align*}
\end{definition}

\begin{definition}\label{D:Ln}
   Formally define the functions $\mathcal{L}_n: \R_+ \times \mathbb{T}^{4d} \to
   \R_+$ recursively by
  \begin{align}\label{E:Ln}
    \mathcal{L}_n\left(t,x_0,x,x_0',x'\right) \coloneqq
    \begin{dcases}
      G\left(t,x_0,x\right)G\left(t,x_0',x'\right)                                            & \text{if $n = 0$,} \\
      \left(\mathcal{L}_0\triangleright \mathcal{L}_{n-1}\right) \left(t,x_0,x,x_0',x'\right) & \text{otherwise,}
    \end{dcases}
  \end{align}
  and for $\lambda\ne 0$, define the resolvent $\mathcal{K}: \R_+ \times
  \mathbb{T}^{4d} \to \R_+$ by
  \begin{align}\label{E:K}
    \mathcal{K}_\lambda\left(t,x_0,x,x_0',x'\right) \coloneqq
    \sum_{n=0}^{\infty} \lambda^{2n} \mathcal{L}_n\left(t,x_0,x,x_0',x'\right).
  \end{align}
\end{definition}

The aim of this section is to prove the following proposition, which shows the
well-posedness of $\mathcal{L}_n$ and $\mathcal{K}_\lambda$ and provides some
estimates at the same time:

\begin{proposition}\label{P:LK}
  There exists a constant $C_{\alpha,\rho,d}>0$ such that for all $t>0$,
  $x,x_0,x',x_0'\in \mathbb{T}^d$, $\lambda\ne 0$, and $n\ge 1$, it holds that
  \begin{align}\label{E:Ln-h}
    \mathcal{L}_n(t,x_0,x,x_0',x')
    \leq C_{\alpha,\rho,d}^n\, G(t,x_0,x)G(t,x_0',x')  h_n(t),
  \end{align}
  and, by denoting $C_*\coloneqq \lambda\, C_{\alpha,\rho,d}^{1/2}\,$,
  \begin{align}\label{E:KH}
    \mathcal{K}_\lambda\left(t,x_0,x,x_0',x'\right)
    \le G(t,x_0,x)G(t,x_0',x') H_{C_*}(t) <\infty.
  \end{align}
\end{proposition}
\begin{proof}
  Notice that we can use the bridge density to rewrite $\mathcal{L}_1$ as
  follows:
  \begin{align}\label{E:I1-I2}
    \mathcal{L}_1(t,x_0,x,x_0',x') =
              & G(t,x_0,x)G(t,x_0',x')\left(\int_{0}^{t/2}\ud s + \int_{t/2}^{t}\ud s \right) \nonumber                        \\
              & \times\iint_{\mathbb{T}^{2d}} \ud z\ud z'\; G_{t,x_0,x}(s,z)G_{t,x_0',x'}(s,z')f_{\alpha,\rho}(z,z') \nonumber \\
    \eqqcolon & G(t,x_0,x)G(t,x_0',x') \left(I_1+I_2\right).
  \end{align}
  By the symmetry of the Brownian bridge, we only need to estimate $I_1$:
  \begin{align}\label{E:I1-only}
    \mathcal{L}_1(t,x_0,x,x_0',x') =  2 G(t,x_0,x)G(t,x_0',x') I_1.
  \end{align}

  The proof below consists of five steps. Fix an arbitrary $\epsilon>0$ and we
  will divide the proof into two cases corresponding to large time $t\ge
  \epsilon$ and small time $t <\epsilon$ in the first two steps, and the value of
  $\epsilon$ will be determined in the third step of the proof. We will use $C$
  to denote a generic constant that may change values at each appearance.
  \bigskip

  \paragraph{\textbf{Step 1.~The case for $\mathcal{L}_1$ with $t\ge
  \epsilon$.~}~} In this case, we are going to use the fact that the fundamental
  solution $G$ is bounded above and below uniformly for all $t\ge \epsilon$. To
  this end, first observe that by Lemma~\ref{L:Cov}, there exists a $\rho_*>0$
  large enough such that $f_{\alpha,\rho_*}(x,y)\ge 0$ for all
  $x,y\in\mathbb{T}^d$. Let $\rho_*$ be the smallest such $\rho_*$, i.e.,
  \begin{align*}
    \rho_* \coloneqq
    \inf\left\{\rho:\, f_{\alpha,\rho}(x,y)\ge 0,\,\,\forall x,y\in \mathbb{T}^d \right\}.
  \end{align*}
  It is ready to see that
  \begin{align}\label{E:hatRho}
      0\leq|f_{\alpha,\rho}(x,y)|\leq f_{\alpha,\widehat{\rho}}(x,y),  \quad
      \text{with $\widehat{\rho} \coloneqq \rho \vee \rho_*$.}
  \end{align}
  Since $t\ge \epsilon$, we can apply lemma~\ref{L:Large} to yield that
  \begin{align*}
    I_1 & = |I_1| = \left|\int_{0}^{t/2}\ud s \iint_{\mathbb{T}^{2d}} \ud z\ud z'\; G_{t,x_0,x}(s,z)G_{t,x_0',x'}(s,z')f_{\alpha,\rho}(z,z')\right| \\
        & \leq C_\epsilon^d \int_{0}^{t/2}\ud s \iint_{\mathbb{T}^{2d}} \ud z\ud z'\; G(s,x_0,z)G(s,x_0',z')|f_{\alpha,\rho}(z,z')|                 \\
        & \leq C_\epsilon^d \int_{0}^{t/2}\ud s \iint_{\mathbb{T}^{2d}} \ud z\ud z'\; G(s,x_0,z)G(s,x_0',z')f_{\alpha,\widehat{\rho}}(z,z')
          \eqqcolon C_\epsilon^d\, I_*,
  \end{align*}
  where we have applied~\eqref{E:hatRho} in the second inequality. We can
  evaluate $I_*$ using the Fourier series representations of $G$ and
  $f_{\alpha,\widehat{\rho}}$ to see that
  \begin{align*}
    |I_*|= & \left|\frac{1}{(2\pi)^{d/2}} \sum_{k\in \mathbb{Z}^d} \int_{0}^{t/2}\ud s\, e^{2ik\cdot (x_0-x_0')} \mathcal{F}\left(f_{\alpha,\widehat{\rho}}\right)(k) e^{-s|k|^2}\right|                                \\
    \leq   & \frac{1}{(2\pi)^{d/2}} \sum_{k\in \mathbb{Z}^d} \int_{0}^{t/2}\ud s\, \mathcal{F}\left(f_{\alpha,\widehat{\rho}}\right)(k) e^{-s|k|^2}                                                                     \\
    =      & \frac{1}{(2\pi)^{d/2}} \sum_{k\in \mathbb{Z}^d} \int_{0}^{t/2}\ud s\, \left[ \mathcal{F}\left(f_{\alpha,\rho}\right)(k)+ \mathcal{F}\left(\frac{\widehat{\rho}-\rho}{(2\pi)^d}\right)(k)\right ] e^{-s|k|^2} \\
    =      & \frac{1}{(2\pi)^{d/2}} \int_0^{t/2} \left[k_1(s) + (\widehat{\rho}-\rho)\right]\ud s.
  \end{align*}
  Therefore, by enlarging $t/2$ to $t$ in the above integral and taking into
  account of the expression $C_\epsilon$ in~\eqref{E:C-Large}, we have proved
  that when $t \ge \epsilon$,
  \begin{align}\label{E:Case-Large}
    \mathcal{L}_1(t,x_0,x,x_0',x')
    \leq C \left(1+\frac{1}{\sqrt{\epsilon}}\right)^d e^{d\pi^2/\epsilon}
         G(t,x_0,x)G(t,x_0',x') \int_0^t \left[k_1(s) + (\widehat{\rho}-\rho)\right]\ud s.
  \end{align}
  \bigskip

  \paragraph{\textbf{Step 2.~The case for $\mathcal{L}_1$ with $t<\epsilon$.~~}}
  In this case, we embed the torus in $\mathbb{R}^d$; then we apply the
  techniques of~\cite{chen.kim:19:nonlinear} to achieve the upper bound. To this
  end, recall that the Riesz kernel $f_\alpha^*$ and its Fourier transform
  $\widehat{f}_\alpha^*$ on $\R^d$ are specified in~\eqref{E:Riesz}.

  From~\eqref{E:Convention}, we can see that the integrand for $I_1$ given
  in~\eqref{E:I1-I2} are $2\pi$-periodic in each component of $z$ and $z'$.
  Hence, we can equivalently integrate over the domains: $(z,z')\in
  \mathbb{T}^d(x_0) \times \mathbb{T}^d(x_0')$, where
  \begin{align*}
    \mathbb{T}^d(x) \coloneqq [x_1-\pi,x_1+\pi] \times \cdots \times [x_d-\pi,x_d+\pi]
    \quad \text{for all $x=(x_1,\cdots,x_d)\in \mathbb{R}^d$.}
  \end{align*}
  Hence,
  \begin{align*}
    I_1 & = \int_0^{t/2}\ud s
            \int_{\mathbb{T}^d(x_0')}\ud z'
            \int_{\mathbb{T}^d(x_0)} \ud z\;
            G_{t,x_0,x}(s,z)G_{t,x_0',x'}(s,z')f_{\alpha,\rho}(z,z').
  \end{align*}
  Since both $z-x_0$ and $z'-x_0'$ in the above integral belong to
  $\mathbb{T}^d$, we can apply Lemma~\ref{L:BMBridge} to see that
  \begin{align*}
    I_1 \le C(1+t)^d \sum_{k,k'\in \Pi^d} & \int_0^{t/2}\ud s
              \int_{\mathbb{T}^d(x_0')}\ud z'
              \int_{\mathbb{T}^d(x_0)} \ud z\; \\
            & \times p_{t,x_0,x+k}(s,z)p_{t,x_0',x'+k'}(s,z')f_{\alpha,\rho}(z,z'),
  \end{align*}
  where recall that $\Pi \coloneqq \left\{-2\pi,0,2\pi\right\}$.
  By~\eqref{E:fEst}, we see that for some universal constant $C>0$,
  \begin{align*}
    f_{\alpha,\rho}\left(z,z'\right)
    =   f_{\alpha,\rho}\left(\myMod{z-z'}\right)
    \le C \sum_{k''\in \Pi^d} f_{\alpha,\rho}^*\left(z-z'+ k''\right).
  \end{align*}
  Thanks to the nonnegativity of the integrand, we can apply the above
  inequality to $I_1$ and extend the integration domain from
  $\mathbb{T}^d(x_0)\times\mathbb{T}^d(x_0')$ to $\mathbb{R}^{2d}$ to yield that
  \begin{align*}
    I_1 \le C (1+t)^d \sum_{k,k',k''\in \Pi^d}
        & \int_0^{t/2}\ud s \: (1+s)^d
          \iint_{\mathbb{R}^{2d}}\ud z' \ud z\;
          f_{\alpha,\rho}^*\left(z-z'+ k''\right)\\
        & \times
          p_{t,x_0 , x +k }\left(s,z \right)
          p_{t,x_0', x'+k'}\left(s,z'\right).
  \end{align*}
  We have reduced the problem from $\mathbb{T}^d$ to $\mathbb{R}^d$. Now we can
  apply the Plancherel theorem to the above integral to obtain
  \begin{align*}
    I_1 = |I_1|
    \le & C (1+t)^d \int_0^{t/2}\ud s \: \int_{\mathbb{R}^d}\ud \xi\: e^{-\frac{s(t-s)}{t}|\xi|^2} \widehat{f}_\alpha^*(\xi) \\
    \le & C (1+t)^d \int_0^{t/2}\ud s \: \int_{\mathbb{R}^d}\ud \xi\: e^{-\frac{s}{2}|\xi|^2} \widehat{f}_\alpha^*(\xi)      \\
     =  & C (1+t)^d \int_0^{t/2}\ud s \: k_2(s),
  \end{align*}
  where the second inequality is due to the fact that $s/2\le s(t-s)/t$ when
  $s\in [0,t/2]$. Therefore, we have proved that when $t\in (0,\epsilon)$,
  \begin{align}\label{E:Case-Small}
    \mathcal{L}_1(t,x_0,x,x_0',x')
    \leq C \left(1+\epsilon\right)^d G(t,x_0,x)G(t,x_0',x') \int_0^t k_2(s)\ud s.
  \end{align}
  Note that we use the condition $t<\epsilon$ only in the last step to bound the
  factor $(1+t)^d$ by $(1+\epsilon)^d$.\bigskip

  \paragraph{\textbf{Step 3.~Determination of $\epsilon$ for the
  estimate of $\mathcal{L}_1$.~}~} Combining the estimates in
  both~\eqref{E:Case-Large} and~\eqref{E:Case-Small}, we see that
  \begin{align}\label{E:Case-Combined}
    \mathcal{L}_1(t,x_0,x,x_0',x')
    \leq & C  \,(C_{1,\epsilon} + C_{2,\epsilon}) G(t,x_0,x)G(t,x_0',x') \int_0^t \left[k_1(s)+k_2(s)+1\right]\ud s,
  \end{align}
  for all $t>0$, where
  \begin{align*}
    C_{1,\epsilon} \coloneqq \left(1+\frac{1}{\sqrt{\epsilon}}\right)^d e^{d\pi^2/\epsilon} \quad \text{and} \quad
    C_{2,\epsilon} \coloneqq \left(1+\epsilon\right)^d.
  \end{align*}
  It is clear that both $C_{i,\epsilon}$ are continuous functions of
  $\epsilon>0$. Since $C_{1,\epsilon}$ is monotone decreasing while
  $C_{2,\epsilon}$ is monotone increasing, there exists an unique $\epsilon_0>0$
  that minimize the addition of the two, namely, $\epsilon_0\coloneqq \arg\min
  \left(C_{1,\epsilon}+ C_{2,\epsilon}\right).$ Finally, one can simply choose
  the constant $C_{\alpha,\rho,d}$ in~\eqref{E:Ln-h} to be $C_{1,\epsilon_0}+
  C_{2,\epsilon_0}$ (up to another factor of generic constant). This completes
  the proof of~\eqref{E:Ln-h} in case of $n=1$. \bigskip

  \paragraph{\textbf{Step 4.~Upper bounds for $\mathcal{L}_n$.~}~}
  Suppose now~\eqref{E:Ln} holds for $n-1$. By Definition~\ref{D:Ln} and the
  induction assumption, we have
  \begin{align}\label{Ln bound 1}
    \mathcal{L}_n= & \int_0^{t} \ud s \iint_{\mathbb{T}^{2d}} \ud z \ud z' G(t-s,z,x)G(t-s,z',x')\mathcal{L}_{n-1}(s,x_0,z,x_0',z')f_{\alpha,\rho}(z,z')\nonumber \\
    \leq           & C^{n-1} G(t,x_0,x)G(t,x_0',x')\nonumber                                                                                                      \\
                   & \times \int_0^t \ud s\, h_{n-1}(s) \iint_{\mathbb{T}^{2d}} \ud z \ud z' G_{t,x_0,x}(s,z)G_{t,x_0',x'}(s,z')f_{\alpha,\rho}(z,z').
  \end{align}
  Using the fact that $h_n(t)$ is nondecreasing (see Lemma~\ref{L:h-inc}) and by
  the symmetry of Brownian bridge, we see that
  \begin{align*}
    \mathcal{L}_n\left(t,x_0,x,x_0',x'\right)
    \le 2 C^{n-1} G(t,x_0,x)G(t,x_0',x') I_n(t),
  \end{align*}
  where
  \begin{align*}
    I_n(t) \coloneqq
    \int_0^{t/2}\ud s\, h_{n-1}(t-s)
    \iint_{\mathbb{T}^{2d}} \ud z \ud z'\, G_{t,x,x_0}(s,z)G_{t,x',x_0'}(s,z')f_{\alpha,\rho}(z,z').
  \end{align*}
  Now, one can carry out the same three steps as the proof of
  Proposition~\ref{P:LK} to bound $I_n(t)$:
  \begin{align*}
    I_n(t) \leq & C \int_0^{t/2} h_{n-1}(t-s)\left(k_1(s)+k_2(s)+1\right) \ud s \\
           \leq & C \int_0^t     h_{n-1}(t-s)\left(k_1(s)+k_2(s)+1\right) \ud s
           =      C h_n(t).
  \end{align*}
  This completes the proof of~\eqref{E:Ln-h}.\bigskip

  \paragraph{\textbf{Step 5.~Upper bound for the resolvent
  \texorpdfstring{$\mathcal{K}$}{}}~} Finally, the estimate of $\mathcal{K}$
  in~\eqref{E:KH} is a direct consequence of the estimate of $\mathcal{L}_n$
  in~\eqref{E:Ln-h} and the definition of $H_\lambda(t)$ in~\eqref{E:Ht}. This
  completes the whole proof of Proposition~\ref{P:LK}.
\end{proof}

\section{Proof of Theorem~\ref{T:Main}} \label{S:Main T}

Now we are ready to introduce introduce the mild solution to~\eqref{E:PAM} and
establish its well-posedness. Let $\dot{W}$ be the centered and spatially
homogeneous Gaussian noise introduced above, defined on a complete probability
space $(\Omega, \mathcal{F}, \P)$.  Let $\left\{W_t(A); t \ge 0, A \in
\mathscr{B}(\dom^d)\right\}$ be the martingale measure associated to the noise
$\dot W$ in the sense of Walsh~\cite{walsh:86:introduction}, where
$\mathscr{B}\left(\dom^d\right)$ refers to the Borel $\sigma$-algebra on
$\dom^d\subseteq\R^d$. Let $\{\mathcal{F}_t\}_{t\ge 0}$ be the underlying
augmented filtration generated by $\dot{W}$
\begin{align*}
  \mathcal{F}_t = \sigma\left\{W_s(A): 0 \le s \le t, A \in \mathscr{B}(\dom) \right\}\vee \mathcal{N},
\end{align*}
where $\mathcal{N}$ is the $\sigma$-field generated by all $\mathbb{P}$-null
sets in $\mathcal{F}$.

\begin{definition}\label{D:Sol}
  A process $u = \left\{u(t, x);\, t > 0, x \in \dom^d\right\}$ is called
  \textit{a random field solution} or the \textit{mild solution}
  to~\eqref{E:PAM} if:
  \begin{enumerate}
    \item[(i)] $u$ is adapted, i.e., for each $t > 0$ and $x \in \dom^d$, $u(t,
      x)$ is $\mathcal{F}_t$-measurable;
    \item[(ii)] $u$ is jointly measurable with respect to $\mathscr{B}\left((0,
      \infty) \times \dom^d\right) \times \mathcal{F}$;
    \item[(iii)] for each $t > 0$ and $x \in \dom^d$, it holds that
      \begin{align}\label{D:Sol-L2}
        \E \left(\int_0^t \ud s \iint_{\dom^{2d}}\ud y \, \ud y'\,
        G(t-s,x,y) u(s, y)f(y,y') G(t-s, x, y') u(s, y')\right) < \infty;
      \end{align}
    \item[(iv)] for each $t > 0$ and $x \in \dom^d$, $u$
      satisfies~\eqref{E:mild-sol} a.s. for all $(t,x)\in
      (0,\infty)\times\dom^d$.
  \end{enumerate}
\end{definition}

Now we are ready to prove Theorem~\ref{T:Main}.

\begin{proof}[Proof of Theorem~\ref{T:Main}]
  The existence and uniqueness in $L^2(\Omega)$, as well as the two-point
  correlation estimates of the solution to~\eqref{E:PAM}, can be established
  from the standard Picard iteration. More precisely, with
  Proposition~\ref{P:LK}, we can carry out the same six steps as those in the
  proof of Theorem~2.4 in Section~3.3 of~\cite{chen.dalang:15:moments}; One may
  also check the proofs of Theorems~1.4 and~1.5
  of~\cite{candil.chen.ea:23:parabolic}. \smallskip

  It remains to prove the $p$-th moment bounds in~\eqref{E:p-Moments}. One can
  follow the same strategy as in the proof of Theorem~1.7
  of~\cite{chen.huang:19:comparison} to establish~\eqref{E:p-Moments}; See the
  proof of part~(ii) of Theorem~1.4 in Section~5.1
  of~\cite{candil.chen.ea:23:parabolic} for another presentation. In essence,
  let $u_n(t,x)$ be the immature solutions in the Picard iterations, namely,
  $u_0(t,x) = J_0(t,x)$, and for $n\ge 1$,
  \begin{align*}
     u_n(t,x) = J_0(t,x) + \lambda \int_0^t \int_{\dom^d} G(t-s, x, y) u_{n-1}(s, y) W(\ud s,\ud y).
  \end{align*}
  An application of the \textit{Burkholder-Davis-Gundy inequality} (BDG
  inequality) shows that their $p$-th moments satisfy the following integral
  inequality
  \begin{align}\label{E:BDG-Picard}
    \Norm{u_{n+1}(t,x)}_p^2
    \le 2 J_0^2(t,x) + 8p\lambda^2 \int_0^t\ud s\iint_{\mathbb{T}^{2d}} \ud y\ud y' f_{\alpha,\rho}\left(y,y'\right)
            G(t-s,x,y )\Norm{u_n(s,y )}_p & \notag \\
    \times  G(t-s,x,y')\Norm{u_n(s,y')}_p & .
  \end{align}
  Then one can show that
  \begin{align*}
    g_n(t,x)\coloneqq \sqrt{2} J_0(t,x)  \left( \sum_{k=0}^{n} \left[16 p \lambda^2\right]^k h_k(t) \right)^{1/2}
  \end{align*}
  is a super solution to~\eqref{E:BDG-Picard}, namely, $\Norm{u_{n}(t,x)}_p \le
  g_n(t,x)$. Then one can show that for $(t,x)$ fixed, $\{u_n(t,x): n\ge 0\}$ is
  a Cauchy sequence in $L^p(\Omega)$. Finally, the memorization relation holds
  in the limit:
  \begin{align*}
    \Norm{u(t,x)}_p
    =   \lim_{n \to \infty} \Norm{u_n(t,x)}_p
    \le \lim_{n \to \infty} g_n(t,x)
    =  \sqrt{2} J_0(t,x) \left( \sum_{k=0}^{\infty} \left[16 p \lambda^2\right]^k h_k(t) \right)^{1/2},
  \end{align*}
  where $h_k(t)$ are given in~\eqref{E:hn}. This completes the sketch proof
  of~\eqref{E:p-Moments}. We refer the interested readers to the proof of
  Theorem~1.7 of~\cite{chen.huang:19:comparison} or the proof of part~(ii) of
  Theorem~1.4 in Section~5.1 of~\cite{candil.chen.ea:23:parabolic} for more
  details.

  Finally, when $\lambda^2 p$ is large enough, one can
  obtain~\eqref{E:p-Mom-Rate} from the estimate of $\gamma_0(\lambda)$
  in~\eqref{E:LogH-gamma}. This completes the whole proof of
  Theorem~\ref{T:Main}.
\end{proof}

\section{Lower bound for the second moment}\label{S:Lbd}

\begin{proof}[Proof of Theorem~\ref{T:Lbd}]
  Fix an arbitrary $\epsilon > 0$. Under condition~\eqref{E:Cf}, thanks to
  Lemma~\ref{L:Large}, one can derive the corresponding lower bound in Step~1 of
  the proof of Proposition~\ref{P:LK}. In particular, when $t\ge \epsilon$,
  \begin{align*}
    \mathcal{L}_1 (t,x_0,x,x_0',x')
    & \ge c_\epsilon^d \int_{0}^{t/2}\ud s \iint_{\mathbb{T}^{2d}} \ud z\ud z'\; G(s,x_0,z)G(s,x_0',z')f(z,z') \\
    & \ge c_\epsilon^d C_f \int_{0}^{t/2}\ud s \iint_{\mathbb{T}^{2d}} \ud z\ud z'\; G(s,x_0,z)G(s,x_0',z')    \\
    & = c_\epsilon^d C_f \frac{t}{2},
  \end{align*}
  where the constant $c_\epsilon$ is given in~\eqref{E:C-Large}. A lower bound
  of $\mathcal{L}_2$ can be derived similarly:
  \begin{align*}
    & \mathcal{L}_2\left(t,x_0,x,x_0',x'\right)                                                                                      \\
    & =\int_0^{t} \ud s \iint_{\mathbb{T}^{2d}} \ud z \ud z' G(s,z,x)G(s,z',x')\mathcal{L}_1(t-s,x_0,z,x_0',z')f(z,z')               \\
    & \geq \int_0^{\frac{t}{2}} \ud s \iint_{\mathbb{T}^{2d}} \ud z \ud z' G(s,z,x)G(s,z',x')\mathcal{L}_1(t-s,x_0,z,x_0',z')f(z,z') \\
    & \geq \frac{c_\epsilon^dC_f}{2}\int_0^{\frac{t}{2}} \ud s\, (t-s)\iint_{\mathbb{T}^{2d}} \ud z \ud z' G(s,z,x)G(s,z',x')\; C_f  \\
    & =  \frac{c_\epsilon^dC_f^2}{2}\frac{1}{2}\left(t^2-\left(\frac{t}{2}\right)^2\right)
      \geq \frac{c_\epsilon^dC_f^2}{2}\frac{1}{2}\frac{t^2}{2}.
  \end{align*}
  Finally, an induction argument together with the elementary relation
  $t^n-({t/2})^n\geq t^n/2$ gives us that, uniformly for all $x_0,x,x_0',x'\in
  \mathbb{T}^d$,
  \begin{align*}
    \mathcal{L}_n\left(t,x_0,x,x_0',x'\right)
    \geq \frac{c_\epsilon^d}{2}\frac{1}{n!}\left(\frac{C_ft}{2}\right)^n,
  \end{align*}
  and hence, for any $\lambda\ne 0$,
  \begin{align*}
    \mathcal{K}_\lambda\left(t,x_0,x,x_0',x'\right)
    \ge \frac{c_\epsilon^d}{2}\exp\left(\frac{C_f t}{2}\right).
  \end{align*}
  Then an application of~\eqref{E:TwoPoint} proves Theorem~\ref{T:Lbd}.
\end{proof}

When the initial condition is given by a bounded measurable function, one has
the Feynman-Kac representation for second moment of the solution to the
parabolic Anderson model (cf.~\cite{hu.nualart.ea:11:feynman-kac}). Taking
advantage of this Feynman-Kac representation, we are able to prove the
exponential lower bound of the second moment for all $\rho>0$.

\begin{proof}[Proof of Theorem~\ref{T:Lbd F-K}]
   When the initial condition is given by a bounded measurable function $f$, one
   has the Feynman-Kac formula for the second moment in the following form,
  \begin{align}\label{E:FeynmanKac}
    \E\left[u(t,x)^2\right] = \E_x\left[f(B_s)f(\tilde{B}_s)\exp\left\{\int_0^tf_{\alpha,\rho}(B_s, \widetilde{B}_s) \ud s \right\}\right],
  \end{align}
  where $B_s$ and $\widetilde{B}_s$ are two independent Brownian motions on
  torus starting from $x$, and $\E_x$ is taking expectation of the randomness of
  the two Brownian motions. Since we assumed that $f$ is also bounded below away
  from zero, it is clear
  \begin{align*}
    \mathbb{E}\left[u(t,x)^2\right] \gtrsim \E_x\left[\exp\left\{\int_0^tf_{\alpha,\rho}(B_s, \widetilde{B}_s) \ud s \right\}\right].
  \end{align*}
  An application of Jensen's inequality give
 \begin{align}\label{E:FeynmanKac2}
    \mathbb{E}\left[u(t,x)^2\right]
    \gtrsim \left[\exp\left\{\int_0^t\E_x \left(f_{\alpha,\rho}\left(B_s, \widetilde{B}_s\right) \right)\ud s \right\}\right].
  \end{align}
  Now recall the definition of $f_{\alpha,\rho}$ in~\eqref{E:f_alpha}, the
  spectral representation of the heat kernel $G(t,x,y)$ in~\eqref{E:HeatSpec},
  together with the fact that
  \begin{align*}
      \E_x\left((2\pi)^{-d/2}e^{ik\cdot B_s}\right)
    = \E_x\left((2\pi)^{-d/2}e^{-ik\cdot \tilde{B}_s}\right)
    = e^{-\frac{|k|^2}{2}s},
  \end{align*}
  the exponent on the right hand-side of~\eqref{E:FeynmanKac2} equals
  \begin{align*}
    \left[\int_0^t \E_x \left(f_{\alpha,\rho}\left(B_s,\tilde{B}_s\right)\right) \ud s\right]
     =    & \frac{\rho\, t}{(2\pi)^d} + \frac{2^\alpha}{(2\pi)^d}\sum_{k\in \mathbb{Z}_*^d} |k|^{-2\alpha} \int_0^t e^{-{|k|^2}s}\ud s   \\
     =    & {\frac{\rho\, t}{(2\pi)^d} + \frac{2^\alpha}{(2\pi)^d}\sum_{k\in \mathbb{Z}_*^d}} |k|^{-2\alpha-2}\left[1-e^{-|k|^2t}\right] \\
     \sim & \frac{\rho\, t}{(2\pi)^d},\quad \mathrm{as}\ \  t\uparrow \infty.
  \end{align*}
 The proof is thus completed.
\end{proof}

The exponential lower bound suggests full intermittency for the solution. As
explained in the introduction, the exponential growth of the second moment is
due to ergodicity of Brownian motion on a compact manifold: the time average
should converge to the space average
\begin{align*}
  \frac{1}{t}\int_0^t f_{\alpha,\rho}(B_s-\widetilde{B}_s)\ud s
  \to \int_{\mathbb{T}^d} f_{\alpha,\rho} (x) \ud x = \frac{\rho}{(2\pi)^d},
\end{align*}
in $\limsup$ at a rate of $\sqrt{\log\log t}$ (see~\cite{brosamler:83:laws}). It
is clear that our argument above is quite generic and does not depend on the
fact that the state space considered here is a torus (as opposed to a general
compact manifold).

It is well known that there is a phase transitions in the intensity parameter
$\lambda$ (in terms of the growth rate of the second moment) in $\mathbb{R}^d$
for $d\geq 3$ (see, e.g., \cite{chen.kim:19:nonlinear}); a result of the fact
that Brownian motion is transient when $d\geq3$. We do not see this transition
on torus as long as $\rho>0$ and regardless of the value of $\alpha$: the second
moment also blows up exponentially in time. However, it is not clear what
happens when $\rho=0$.

\section{H{\"o}lder Continuity}\label{S:Holder}

We first establish the following lemma, the proof of which follows similar idea
as that of Lemma~3.1 of~\cite{chen.huang:19:comparison}.

\begin{lemma}\label{L:Holder}
  There exists some universal constant $C>0$ such that for all $\beta \in
  \left(0,1\right]$, $x,y\in\mathbb{T}^d$, and $t'\geq t > 0$,
  \begin{align}\label{eq:Holder-T}
    \left|G(t,x)-G(t',x)\right| & \leq Ct^{-\beta/2}\, G(2t',x) \left(t'-t\right)^{\beta/2} \quad \text{and} \\
    \left|G(t,x)-G(t, y)\right| & \leq Ct^{-\beta/2}\, \left[ G(2t,x) + G(2t,y)  \right] \myDist{x,y}^\beta,
    \label{eq:Holder-S}
  \end{align}
  where we recall that $\myDist{x,y} = |\myMod{x-y}|$.
\end{lemma}
\begin{proof}
  In the proof, we use $C$ to denote a generic constant which may change its
  value at each appearance. Choose and fix an arbitrary $\beta\in(0,1]$. We
  start with ~\eqref{eq:Holder-S}. From the mean value theorem, we have for some
  $\xi \coloneqq \xi(x,y) \in \mathbb{T}^d$ such that 
  \begin{align*}
    \left|G(t,x)- G(t,y)\right|
    & = |\nabla G(t,\xi)| \times {\left|\myMod{x-y}\right|}                                                      \\
    & = \left[\sum_{k\in\mathbb{Z}^d}(2\pi t)^{-d/2}\frac{|\xi+2\pi k|}{t}e^{-\frac{|\xi+2\pi k|^2}{2t}}\right] \times \myDist{x,y} \\
    & \leq \left[\frac{C}{\sqrt{t}}\sum_{k\in\mathbb{Z}^d}(4\pi t)^{-d/2}e^{-\frac{|\xi+2\pi k|^2}{4t}}\right]  \times \myDist{x,y} \\
    & =\frac{C}{\sqrt{t}}G(2t,\xi)\myDist{x,y}                                                                                      \\
    & \leq\frac{C}{\sqrt{t}} \left[G(2t,x) + G(2t,y)\right] \myDist{x,y},
  \end{align*}
  where the last inequality can be obtained via scaling arguments (see, e.g.,
  the proof of Lemma~3.1 of~\cite{chen.huang:19:comparison}).
  We can apply the above inequality and~\eqref{E:G-UnifBd} to see that
  \begin{align*}
    |G(t,x)- G(t,y)|
    & = |G(t,x)- G(t,y)|^\beta|G(t,x)- G(t,y)|^{1-\beta}                                                                                 \\
    & \leq \frac{C}{t^{\beta/2}} \left[G(2t,x) + G(2t,y)\right]^{\beta d}\myDist{x,y}^\beta \left[G(2t,x) + G(2t,y)\right]^{(1-\beta) d} \\
    & =\frac{C}{t^{\beta/2}} \left[G(2t,x) + G(2t,y)\right] \myDist{x,y}^\beta,
  \end{align*}
  which proves that~\eqref{eq:Holder-S}. To prove~\eqref{eq:Holder-T}, observe
  that
  \begin{align*}
    |G(t,x)-G(t',x)|
    & \leq \left|\sum_{k\in\mathbb{Z}^d}(2\pi t)^{-d/2}e^{-\frac{|x+2\pi k|^2}{2t}}-(2\pi t')^{-d/2}e^{-\frac{|x+2\pi k|^2}{2t}} \right. \\
    & \qquad \left.+(2\pi t')^{-d/2}e^{-\frac{|x+2\pi k|^2}{2t}} -(2\pi t')^{-d/2}e^{-\frac{|x+2\pi k|^2}{2t'}} \right|                  \\
    & \leq t^{d/2}\left|(t')^{-d/2} -t^{d/2} \right|G(t,x)
      + \left| \sum_{k\in\mathbb{Z}^d} (2\pi t')^{-d/2} \left(e^{-\frac{|x+2\pi k|^2}{2t}}-e^{-\frac{|x+2\pi k|^2}{2t'}}\right)\right|    \\
    & \eqqcolon I_1 + I_2.
  \end{align*}

  Following inequality~(3.3) in~\cite{chen.huang:19:comparison}, we find that
  $I_1 \le C t^{-\beta/2} |t'-t|^{\beta/2} G(t,x)$. As for $I_2$, from the mean
  value theorem, we can deduce that for some $\xi \in [t,t']$,
  \begin{align*}
    I_2 & =   \sum_{k\in\mathbb{Z}^d} (2\pi t')^{-d/2} \frac{|x+2\pi k|^2}{(2\xi)^2} \left|t'-t\right| e^{-\frac{|x+ 2\pi k|^2}{2 \xi}}
          \le C \sum_{k\in\mathbb{Z}^d} (2\pi t')^{-d/2} \frac{1}{2\xi} \left|t'-t\right| e^{-\frac{|x+ 2\pi k|^2}{4 \xi}} \\
        & \le C \sum_{k\in\mathbb{Z}^d} (2\pi t')^{-d/2} \frac{1}{2\xi} \left|t'-t\right| e^{-\frac{|x+ 2\pi k|^2}{4 \xi}}
      \le C \frac{|t'-t|}{t} G(2t',x).
  \end{align*}
  Together with the fact that $I_2 \le G(t',x)$, we see that
  \begin{align*}
    I_2 = I_2^{\beta/2} I_2^{1-\beta/2}
    \le C\left[ \frac{|t'-t|}{t} G(2t',x)\right]^{\beta/2} G(t',x)^{1-\beta/2}
    \le C t^{-\beta/2}|t'-t|^{\beta/2} G(2t',x).
  \end{align*}
  Combining the bounds for $I_1$ and $I_2$ proves~\eqref{eq:Holder-T}.
\end{proof}


Now we are ready to prove Theorem~\ref{T:Holder}.

\begin{proof}[Proof of Theorem~\ref{T:Holder}]
  We need only control the H\"older modulus of
  \begin{align*}
    I(t,x) \coloneqq \iint_{(0,t]\times \mathbb{T}^d}G(t-s,x,y) \lambda u(s,y) W(\ud s,\ud y).
  \end{align*}
  Fix an arbitrary $n\ge 1$. Following~\cite{chen.huang:19:comparison}, we need
  to compute the $p$-th moment increments
  \begin{align*}
    \Norm{I(t,x)-I(t',x')}_p^2 \le C \left[ I_1\left(t,x,x'\right) + I_2\left(t,t',x'\right) + I_3\left(t,t',x\right)\right],
  \end{align*}
  for $t,t'\in [1/n,n]$ and $x,x' \in \mathbb{T}^d$ with $t'>t$, where $I_1$,
  $I_2$, and $I_3$ are defined as follows:
  \begin{align*}
    I_1(t,x,x' ) \coloneqq & \int_0^t \ud s \iint_{\mathbb{T}^{2d}} \ud y_1 \ud y_2 \: f_{\alpha,\rho}(y_1,y_2) \left|G(t-s, x,y_1) -G(t-s,x',y_1)\right| \\
                           & \times \left|G(t-s,x,y_2)-G(t-s,x',y_2)\right| \Norm{u(s,y_1)}_p \Norm{u(s,y_2)}_p;                                          \\
    I_2(t,t',x') \coloneqq & \int_0^t \ud s \iint_{\mathbb{T}^{2d}} \ud y_1 \ud y_2 \: f_{\alpha,\rho}(y_1,y_2)\left|G(t-s,x',y_1)-G(t'-s,x',y_1)\right|  \\
                           & \times \left|G(t-s,x',y_2)-G(t'-s,x',y_2)\right| \Norm{u(s,y_1)}_p \Norm{u(s,y_2)}_p;                                        \\
    I_3(t,t',x') \coloneqq & \int_t^{t'} \ud s \iint_{\mathbb{T}^{2d}} \ud y_1 \ud y_2 \: f_{\alpha,\rho}(y_1,y_2) G(t'-s,x',y_1)                         \\
                           & \times G(t'-s,x',y_2) \Norm{u(s,y_1)}_p \Norm{u(s,y_2)}_p.
  \end{align*}
  In the following, we use $C_n$ to denote a generic constant that may depend on
  $n$ and may change its value at each occurrence.

  To control $I_1$, from the $p$-th moment formula~\eqref{E:p-Moments}, we see
  that
  \begin{align}\label{E:Holder-I_1}
    \begin{aligned}
    I_1(t,x,x') \leq
    & C_n \int_0^t \ud s \iint_{\mathbb{T}^{2d}} \ud y_1 \ud y_2\: f_{\alpha,\rho}(y_1,y_2)\iint_{\mathbb{T}^{2d}}\mu(\ud z_1) \mu(\ud z_2) \\
    & \times G(s,y_1,z_1)\left| G(t-s,x,y_1) - G(t-s,x',y_1) \right|                                                                        \\
    & \times G(s,y_2,z_2)\left| G(t-s,x,y_2) - G(t-s,x',y_2) \right|.
    \end{aligned}
  \end{align}
  We can then apply~\eqref{eq:Holder-S} to obtain that for all
  $\beta\in (0,1)$,
  \begin{align*}
    \MoveEqLeft \left|G(t-s,x,y_1) - G(t-s,x',y_1)\right| \\
    & = \left|G\left(t-s,\myMod{x-y_1}\right) - G\left(t-s,\myMod{x'-y_1}\right)\right|                                                    \\
    & \le C \left[G\left(2(t-s),\myMod{x-y_1}\right)+G\left(2(t-s),\myMod{x'-y_1}\right)\right]\frac{\myDist{x,x'}^\beta}{(t-s)^{\beta/2}} \\
    & =   C \left[G\left(2(t-s),x,y_1\right)+G\left(2(t-s),x',y_1\right)\right]\frac{\myDist{x,x'}^\beta}{(t-s)^{\beta/2}},
  \end{align*}
  where we have used the fact that
  \begin{align*}
    \myDist{\myMod{x-y_1}, \myMod{x'-y_1}}
    = \left|\myMod{\:\myMod{x-y_1} - \myMod{x'-y_1}\:}\right|
    = \left|\myMod{x-x'}\right|
    = \myDist{x,x'}.
  \end{align*}
Hence,
  \begin{align*}
    G(s,y_1,z_1)
    & \left| G(t-s,x,y_1) - G(t-s,x',y_1)\right|                                                                  \\
    & \le C G(2s,y_1,z_1)\left[G(2(t-s),x,y_1)+G(2(t-s),x',y_1)\right]\frac{\myDist{x,x'}^\beta}{(t-s)^{\beta/2}} \\
    & = C \frac{\myDist{x,x'}^\beta}{(t-s)^{\beta/2}} \left[G(2t,x,z_1) G_{2t,x,z_1}\left(2s,y_1\right) + G(2t,x',z_1) G_{2t',x',z_1}\left(2s,y_1\right) \right],
  \end{align*}
  where we have used the density for the pinned Brownian motion;
  see~\eqref{E:BMBridge-T}. Therefore, by plugging the above upper bound and the
  corresponding one with $(y_1,z_1)$ replaced by $(y_2,z_2)$ back
  to~\eqref{E:Holder-I_1}, we obtain four terms in the expansion:
  \begin{align*}
    I_1(t,x,x') \leq \sum_{k=1}^4 I_{1,k}(t,x,x').
  \end{align*}
  For $I_{1,1}$, we have
  \begin{align*}
    I_{1,1}(t,x,x') \le
    & C_n \myDist{x,x'}^{2\beta} J_0(2t,x) J_0(2t,x) \int_0^t \ud s\: \frac{1}{(t-s)^{\beta}} \\
    & \times \iint_{\mathbb{T}^{2d}}\ud y_1\ud y_2 \: f_{\alpha,\rho}(y_1,y_2) G_{2t,x,z_1}\left(2s,y_1\right)  G_{2t,x,z_1}\left(2s,y_2\right).
  \end{align*}
  Now an application of Lemma~\ref{L:BMBridge} shows that
  \begin{align*}
    I_{1,1}(t,x,x') \le
    & C_n \myDist{x,x'}^{2\beta} J_0(2t,x) J_0(2t,x) \int_0^t \ud s\: \frac{1}{(t-s)^{\beta}} \\
    & \times \iint_{\mathbb{T}^{2d}}\ud y_1\ud y_2 \: f_{\alpha,\rho}(y_1,y_2) \sum_{k,k'\in\Pi^d} p_{2t,x,z_1+k}\left(2s,y_1\right) p_{2t,x,z_1+k'}\left(2s,y_2\right).
  \end{align*}
  By the same arguments as the Step~2 of the proof of Theorem~\ref{T:Main}, we
  obtain that
  \begin{align*}
    I_{1,1}(t,x,x')
    \le & C_n \myDist{x,x'}^{2\beta} J_0^2(2t,x) \int_0^t \ud s\: \frac{1}{(t-s)^{\beta}} \iint_{\R^{2d}}\ud y_1\ud y_2 \: f_{\alpha,\rho}^*(y_1-y_2+k'') \\
        & \times \sum_{k,k',k''\in\Pi^d} p_{2t,x,z_1+k}\left(2s,y_1\right) p_{2t,x,z_1+k'}\left(2s,y_2\right)                                               \\
    \le & C_n \myDist{x,x'}^{2\beta} J_0^2(2t,x) \int_0^t \frac{k_2(s)}{s^{\beta}} \ud s                                                                  \\
     =  & C_n \myDist{x,x'}^{2\beta} J_0^2(2t,x) \int_0^t s^{\alpha-d/2 -\beta}\ud s,
  \end{align*}
  where the last step is due to~\eqref{E:k2}. Therefore, provided that $\beta <
  1+\alpha -d/2 \in(0,1)$ (recall that $\alpha\in (0,d/2)$), we have
  \begin{align*}
    I_{1,1}(t,x,x') \le C_n \myDist{x,x'}^{2\beta} J_0^2(2t,x).
  \end{align*}
  The computations for the $I_{1,2}, I_{1,3}, I_{1,4}$ are similar. Combining
  all these bounds, we conclude that
  \begin{align*}
    I_1(t,x,x') \le C_n \left( J_0(2t,x) + J_0(2t,x') \right)^2 \myDist{x,x'}^{2\beta}.
  \end{align*}
  The proof for the time increments, namely, $I_2$ and $I_3$, is similar, which
  will be omitted here.
\end{proof}

\appendix

\section{Proof of Lemma~\ref{L:G-bdd}}\label{A:G-bdd}

\begin{proof}[Proof of Lemma~\ref{L:G-bdd}]
  Noticing that
  \begin{align*}
    \frac{G_d(t,x)}{p_{d}(t,x)} = \prod_{i=1 }^{d}  \frac{G_1(t,x_i)}{p_1(t,x_i)},
  \end{align*}
  to prove~\eqref{E:G-bdd}, we only need to prove the case $d=1$. In the
  following, we assume that $d=1$. It is known that the fundamental solution to
  the one-dimensional heat equation on a compact set can be expressed using the
  \textit{Theta functions} $\theta_i(\cdot)$, $i=1,2,3,4$. In particular, when
  the periodic boundary condition is imposed, the fundamental solution can be
  expressed using the $\theta_3$ function (see (20.2.3) on p.~524
  of~\cite{olver.lozier.ea:10:nist}):
  \begin{align}\label{E:theta-q}
    \theta_3(z,q) \coloneqq 1+ 2 \sum_{n=1}^{\infty} q^{n^2} \cos\left(2n z\right),
    \quad \text{for $z\in \mathbb{C}$ and $q\in (0,1)$}.
  \end{align}
  There is another parameterization using $\tau \in \mathbb{C}$ with
  $\text{Im}(\tau)>0$:
  \begin{align}\label{E:theta-tau}
    \theta_3(z|\tau) \coloneqq \theta_3\left(z, q^{\pi i \tau}\right).
  \end{align}
  From (20.13.4) on p.~534 ({\it ibid.}), the reflection formula (20.7.32) on
  p.~531 ({\it ibid.}), and the relation between the two parameterizations
  in~\eqref{E:theta-tau}, we see that the fundamental solution $G(t,x)$ can be
  written in terms of $\theta_3$ in two equivalent ways (each having two
  parameterizations):
  \begin{alignat}{2}\label{E:G-theta}
    G(t,x) & = \frac{1}{2\pi}\theta_3\left( \frac{x}{2}, e^{-\frac{t}{2}}\right)   & = & \frac{1}{\sqrt{2\pi t}} e^{-\frac{x^2}{2t}} \theta_3\left( i \frac{\pi x}{t}, e^{-\frac{2\pi^2}{t}}\right) \\
           & =\frac{1}{2\pi}\theta_3\left(\frac{x}{2}\bigg| \frac{it}{2\pi}\right) & = & \frac{1}{\sqrt{2\pi t}} e^{-\frac{x^2}{2t}} \theta_3\left( i \frac{\pi x}{2}\bigg| \frac{2 \pi i}{t}\right).
  \end{alignat}
  Now applying the definition of $\theta_3$ in~\eqref{E:theta-q} to the two
  expressions of $G(t,x)$ in~\eqref{E:G-theta}, we see that
  \begin{align}\label{E:G-1}
    G(t,x) & = \frac{1}{2\pi} + \frac{1}{\pi}\sum_{n=1}^{\infty} e^{-\frac{n^2t}{2}} \cos\left(nx\right) \\
           & = \frac{1}{\sqrt{2\pi t}} e^{-\frac{x^2}{2t}} \left(1+2 \sum_{n=1}^{\infty} e^{-\frac{2\pi^2n^2}{t}} \cosh\left(\frac{\pi n x}{t}\right)\right). \label{E:G-2}
  \end{align}

  From~\eqref{E:G-2}, we see that $G(t,x)$ is bounded from below by
  \begin{align*}
    G(t,x)
    \ge \frac{1}{\sqrt{2\pi t}} e^{-\frac{x^2}{2t}} \left(1+2 \sum_{n=1}^{\infty} e^{-\frac{2\pi^2n^2}{t}} \right)
    = p_1(t,x)\: \theta_3\left(0,e^{-\frac{2\pi^2}{t}}\right)
    \eqqcolon  p_1(t,x)\: C_t.
  \end{align*}
  where $C_t \coloneqq \theta_3\left(0,e^{-(2\pi^2)/t}\right)$. It is clear that
  $C_t$ is monotone increasing in $t$. The product forms of $C_t$
  in~\eqref{E:Ct} is due to (20.4.4) on p.~529 ({\it ibid.}). By setting $x=0$
  in~\eqref{E:G-theta}, we see that $C_t$ can be equivalently written in two
  ways:
  \begin{align}\label{E:theta-0}
    C_t = \theta_3\left(0,e^{-\frac{2\pi^2}{t}}\right)
        = \sqrt{\frac{t}{2\pi}} \theta_3\left(0, e^{-\frac{t}{2}}\right).
  \end{align}
  As for the upper bound, since $x\in (-\pi,\pi)$, from~\eqref{E:G-2}, we see
  that
  \begin{align*}
    G(t,x) & \le \frac{1}{\sqrt{2\pi t}} e^{-\frac{x^2}{2t}} \left(1+ \sum_{n=1}^{\infty}     e^{-\frac{2\pi^2n^2}{t}} \left(e^{\frac{\pi^2 n}{t}} + 1\right) \right) \\
           & =   p_1(t,x) \:\left(1+ \sum_{n=1}^{\infty} \left( e^{-\frac{2\pi^2n(n-1/2)}{t}} + e^{-\frac{2\pi^2n^2}{t}} \right) \right)                              \\
           & \le p_1(t,x) \:\left(1+ \sum_{n=1}^{\infty} \left( e^{-\frac{2\pi^2(n-1)^2 }{t}} + e^{-\frac{2\pi^2n^2}{t}} \right) \right)                              \\
           & \le 2p_1(t,x)\:\left(1+ \sum_{n=1}^{\infty} e^{-\frac{2\pi^2n^2}{t}} \right)
             =   2p_1(t,x)\:C_t.
  \end{align*}
  As for the asymptotics of $C_t$ in~\eqref{E:Ct-Asym}, from the product form
  $C_t$ in~\eqref{E:Ct}, we see that as $t\to 0$,
  \begin{align*}
    \log\left(\theta_3\left(0,e^{-\frac{2\pi^2}{t}}\right)\right)
    & =    \sum_{n=1}^{\infty} \left(\log\left(1-e^{-\frac{4n\pi^2}{t}}\right) + 2 \log\left(1+e^{-\frac{2(2n-1)\pi^2}{t}}\right)\right) \\
    & \sim \sum_{n=1}^{\infty} \left(-e^{-\frac{4n\pi^2}{t}} + 2 e^{-\frac{2(2n-1)\pi^2}{t}}\right) \sim  2 e^{-(2\pi^2)/t}.
  \end{align*}
  and from~\eqref{E:theta-0}, as $t\to \infty$,
  \begin{align}\label{E_:theta-Inf}
    \log\left(\sqrt{\frac{2\pi}{t}}\theta_3\left(0,e^{-\frac{2\pi^2}{t}}\right)\right)
    & = \log\left(\theta_3\left(0,e^{-\frac{t}{2}}\right)\right)\notag                                                  \\
    & = \sum_{n=1}^{\infty} \left(\log\left(1-e^{-nt}\right) + 2 \log\left(1+e^{-\frac{(2n-1)t}{2}}\right)\right)\notag \\
    & \sim \sum_{n=1}^{\infty} \left(-e^{-nt} + 2 e^{-\frac{(2n-1)t}{2}}\right)  \sim 2 e^{- t/2}.
  \end{align}

  Regarding the bounds in~\eqref{E:Ct-bdd}, the lower one is due to equalities $s$
  and $s'$ in~\eqref{E:Ct}. As for the upper bound, from equality $s'$
  in~\eqref{E:Ct}, we see that
  \begin{align*}
    C_t = \sqrt{\frac{t}{2\pi}} \left(1+ 2 \sum_{n=1}^{\infty} e^{-n^2t/2} \right)
        \le \sqrt{\frac{t}{2\pi}} \left(1+ 2 \int_0^{\infty} e^{-x^2 t/2}\ud x \right)
        =1 + \sqrt{\frac{t}{2\pi}},
  \end{align*}
  where we have used the Riemann integral to approximate the sum. \bigskip

  As for~\eqref{E:G-UnifBd}, by bounding $\cos(\cdot)$ by one in~\eqref{E:G-1}, we
  see that
  \begin{align*}
    G_1(t,x)
    \le \theta_3\left(0,e^{-t/2}\right)
    = \sqrt{\frac{2\pi}{t}}\theta_3\left(0,e^{-2\pi^2/t}\right)
    = \sqrt{\frac{2\pi}{t}} C_t,
  \end{align*}
  where we have applied~\eqref{E:theta-0} in the above equalities. Then raising
  the above inequality to the power $d$ proves the first inequality
  in~\eqref{E:G-UnifBd}. The second one is due to~\eqref{E:Ct-bdd}.

  As for~\eqref{E:G-Inf}, fix an arbitrary $\epsilon>0$. From~\eqref{E:G-1}, we
  see that
  \begin{align}\label{E_:G-Inf}
    \sup_{x\in \mathbb{T}}\left\lvert G_1(t,x) - \frac{1}{2\pi}\right\rvert
    \le \frac{1}{\pi}\sum_{n=1}^{\infty} e^{-\frac{n^2t}{2}}
    = \frac{1}{2\pi}\left(\theta_3\left(0,e^{-t/2}\right)-1\right)
    = \frac{1}{2\pi}\left(\sqrt{\frac{2\pi}{t}} C_t-1\right),
  \end{align}
  where we have used the identity in~\eqref{E:theta-0} in the last equality.
  From the asymptotics in~\eqref{E_:theta-Inf}, we see that
  $\log\left(\sqrt{\frac{2\pi}{t}} C_t\right) = 2e^{-t/2} +
  o\left(e^{-t/2}\right)$ for large $t$. Since the function $t\to
  \sqrt{\frac{2\pi}{t}} C_t$ is continuous on $[\epsilon,\infty]$, there exists
  some constant $\Lambda_\epsilon\ge 2$ such that
  $\log\left(\sqrt{\frac{2\pi}{t}} C_t\right)\le \Lambda_\epsilon e^{-t/2}$ for
  all $t\ge \epsilon$. Therefore,
  \begin{align*}
    \sqrt{\frac{2\pi}{t}} C_t-1
    \le e^{\Lambda_\epsilon e^{-t/2}} -1
    \le    \Lambda_\epsilon e^{-t/2} e^{\Lambda_\epsilon  e^{-t/2}}
    \le    \Lambda_\epsilon e^{\Lambda_\epsilon} e^{-t/2}, \quad \text{for all $t\ge \epsilon$,}
  \end{align*}
  which proves~\eqref{E:G-Inf} with $\Theta_{\epsilon,1} \coloneqq
  \Lambda_\epsilon e^{\Lambda_\epsilon}$. For the case $d\ge 2$, notice that
  \begin{align*}
    \left\lvert G_d(t,x) - \left(\frac{1}{2\pi}\right)^d\right\rvert
    & \le \sum_{i=1}^d \frac{1}{(2\pi)^{i-1}} \left\lvert G_1(t,x_i)- \frac{1}{2\pi} \right\rvert \times \prod_{j=i+1}^{d} G_1(t,x_j) \\
    & \le \sum_{i=1}^d \frac{1}{(2\pi)^{i-1}} \left\lvert G_1(t,x_i)- \frac{1}{2\pi} \right\rvert \times \left(1+\sqrt{\frac{2\pi}{t}}\right)^{d-i},
  \end{align*}
  where we use the convention that $\prod_{\emptyset} \equiv 1$
  and~\eqref{E:G-UnifBd}. Therefore, an application of~\eqref{E_:G-Inf} shows
  that for all $t\ge \epsilon$,
  \begin{align*}
    \sup_{x\in \mathbb{T}^d}\left\lvert G_d(t,x) - \left(\frac{1}{2\pi}\right)^d\right\rvert
    \le \Theta_{\epsilon,1} e^{-t/2} \sum_{i=1}^d \frac{1}{(2\pi)^{i-1}} \left(1+\sqrt{\frac{2\pi}{\epsilon}}\right)^{d-i}.
  \end{align*}
  Hence, \eqref{E:G-Inf} is proved with
  $\Theta_{\epsilon,d}\coloneqq\Theta_{\epsilon,1}\sum_{i=1}^d
  \frac{1}{(2\pi)^{i-1}} \left(1+\sqrt{(2\pi)/\epsilon}\right)^{d-i}.$ This
  completes the whole proof of Lemma~\ref{L:G-bdd}.
\end{proof}

\section*{Acknowledgments}
L.~C. is partially supported by NSF grant DMS-2246850. Both L.~C. (\#~959981)
and C.~O. (\#851792) are partially supported by the collaboration grant from
Simons foundation.

\printbibliography[title = {References}]
\end{document}